%% file: main.tex
\newif\ifarxiv
\arxivtrue

\ifarxiv
\documentclass[11pt]{article}
\else
\documentclass{colt2020} 
\fi

\ifarxiv
\title{Halpern Iteration for Near-Optimal and Parameter-Free\\ Monotone Inclusion and Strong Solutions to Variational Inequalities}
\else
\title[Halpern Iteration for Near-Optimal and Parameter-Free Monotone Inclusion]{Halpern Iteration for Near-Optimal and Parameter-Free\\ Monotone Inclusion and Strong Solutions to Variational Inequalities}
\fi
\usepackage{times}
\input{prologue.tex}

\ifarxiv
\author{Jelena Diakonikolas\thanks{Part of this research was done while the author was a postdoctoral researcher at UC Berkeley.}\\
Department of Computer Sciences, UW-Madison\\
\texttt{jelena@cs.wisc.edu}}
\date{}
\else
\coltauthor{%
 \Name{Jelena Diakonikolas} \Email{jelena@cs.wisc.edu}\\
 \addr Department of Computer Sciences, University of Wisconsin-Madison\\
 1210 W.~Dayton St, Madison, WI 53706
}
\fi

\begin{document}

\maketitle
\begin{abstract}%
    We leverage the connections between nonexpansive maps, monotone Lipschitz operators, and proximal mappings to obtain near-optimal (i.e., optimal up to poly-log factors in terms of iteration complexity) and parameter-free methods for solving monotone inclusion problems. These results immediately translate into near-optimal guarantees for approximating strong solutions to variational inequality problems, approximating convex-concave min-max optimization problems, and minimizing the norm of the gradient in min-max optimization problems. Our analysis is based on a novel and simple potential-based proof of convergence of Halpern iteration, a classical iteration for finding fixed points of nonexpansive maps. Additionally, we provide a series of algorithmic reductions that highlight connections between different problem classes and lead to lower bounds that certify near-optimality of the studied methods.
\end{abstract}
\ifarxiv
\else
\fi
\section{Introduction}

Given a closed convex set $\cu \subseteq \rr^d$ and a single-valued monotone operator $F:\rr^d\to \rr^d$, i.e., an operator that maps each vector to another vector and satisfies:
\begin{equation}\label{eq:mon-op}
    (\forall \vu, \vv\in \rr^d): \quad \innp{F(\vu) - F(\vv), \vu - \vv} \geq 0,
\end{equation}
the \emph{monotone inclusion} problem consists in finding a point $\vu^*$  that satisfies:
\begin{equation}\label{eq:monotone-inclusion}\tag{MI}
\begin{gathered}
\zeros \in F(\vu) + \partial I_{\cu}(\vu), \; \text{ where }\\
I_{\cu}(\vu) = \begin{cases}
0, &\text{ if } \vu \in \cu,\\
\infty, &\text{ otherwise}
\end{cases}
\end{gathered}
\end{equation}
is the indicator function of the set $\cu\subseteq \rr^d$ and $\partial I_{\cu}(\cdot)$ denotes the subdifferential operator (the set of all subgradients at the argument point) of $I_{\cu}$. 

Monotone inclusion is a fundamental problem in continuous optimization that is closely related to  variational inequalities (VIs) with monotone operators, which  model a plethora  of problems in mathematical programming, game theory, engineering, and finance~\cite[Section 1.4]{facchinei2007finite}. Within machine learning, VIs with monotone operators and associated monotone inclusion problems arise, for example, as an abstraction of convex-concave min-max optimization problems, which naturally model adversarial training~\citep{madry2018towards,arjovsky2017wasserstein,arjovsky2016towards,goodfellow2014generative}. 

When it comes to convex-concave min-max optimization,  approximating the associated VI leads to guarantees in terms of the optimality gap. Such guarantees are generally possible only when the feasible set $\cu$ is bounded; a simple example that demonstrates this fact is $\Phi(\vx, \vy) = \innp{\vx, \vy}$ with the feasible set $\vx, \vy \in \rr^d.$ The only (min-max or saddle-point) solution in this case is obtained when both $\vx$ and $\vy$ are the all-zeros vectors. However, if either $\vx \neq \zeros$ or $\vy \neq \zeros$, then the optimality gap $\max_{\vy' \in \rr^d}\Phi(\vx, \vy') - \min_{\vx' \in \rr^d}\Phi(\vx', \vy)$ is infinite.   

On the other hand, approximate monotone inclusion is well-defined even for unbounded feasible sets. In the context of min-max optimization, it corresponds to guarantees in terms of stationarity. Specifically, in the unconstrained setting, solving monotone inclusion corresponds to minimizing the norm of the gradient of $\Phi.$ Note that even in the special setting of convex optimization, convergence in norm of the gradient is much less understood than convergence in  optimality gap~\citep{nesterov2012make,kim2018optimizing}. Further, unlike classical results for VIs that provide 
convergence guarantees for approximating \emph{weak}  solutions~\citep{nemirovski2004prox,nesterov2007dual}, approximations to monotone inclusion lead to approximations to \emph{strong} solutions (see Section~\ref{sec:prelims} for definitions of weak and strong solutions and their relationship to monotone inclusion). 

We leverage the connections between nonexpansive maps, structured monotone operators, and proximal maps to obtain near-optimal algorithms for solving monotone inclusion over different classes of problems with Lipschitz-continuous operators. In particular, we make use of the classical Halpern iteration, which is defined by~\citep{halpern1967fixed}:
\begin{equation}\label{eq:halpern-or}\tag{Hal}
    \vu_{k+1} = \lambda_{k+1}\vu_0 + (1-\lambda_{k+1})T(\vu_k),
\end{equation}
where $T:\rr^d\to\rr^d$ is a nonexpansive map, i.e.,  $\forall \vu, \vv \in \rr^d:$ $\|T(\vu) - T(\vv)\| \leq \|\vu - \vv\|.$ 

In addition to its simplicity, Halpern iteration is particularly relevant to machine learning applications, as it is an \emph{implicitly regularized} method with the following property: if the set of fixed points of $T$ is non-empty, then Halpern iteration~\eqref{eq:halpern-or} started at a point $\vu_0$ and applied with any choice of step sizes $\{\lambda_k\}_{k \geq 1}$ that satisfy all of the following conditions:
\begin{equation}\label{eq:halpern-step-sizes}
    (i) \lim_{k \to \infty} \lambda_k = 0, \quad (ii) \sum_{k=1}^{\infty} \lambda_k = \infty, \quad (iii) \sum_{k=1}^{\infty}|\lambda_{k+1} - \lambda_k| < \infty 
\end{equation}
 converges to the fixed point of $T$ with the minimum $\ell_2$ distance to $\vu_0.$ This result was proved by~\cite{wittmann1992approximation}, who extended a similar though less general result previously obtained by~\cite{browder1967convergence}. The result of \cite{wittmann1992approximation} has since been extended to various other settings~\citep[and references therein]{bauschke1996approximation,xu2002iterative,kohlenbach2011quantitative,kornlein2015quantitative,lieder2019convergence}.
\subsection{Contributions and Related Work}

A special case of what is now known as the Halpern iteration~\eqref{eq:halpern-or} was introduced and its asymptotic convergence properties were analyzed by~\cite{halpern1967fixed} in the setting of $\vu_0 = \zeros$ and $T:\mathcal{B}_2 \to \mathcal{B}_2,$ where $\mathcal{B}_2$ is the unit Euclidean ball. Using the proof-theoretic techniques of~\cite{kohlenbach2008applied}, \cite{leustean2007rates} extracted from the asymptotic convergence result of~\cite{wittmann1992approximation}  the rate at which Halpern iteration converges to a fixed point. The results obtained by~\cite{leustean2007rates} are rather loose and provide guarantees of the form $\|T(\vu_k) - \vu_k\| = O(\frac{M}{\log(k)})$ in the best case (obtained for $\lambda_k = \Theta(\frac{1}{k})$), where $M \geq \|\vu_0\| + \|T(\vu_0)\| + \|\vu_k\|,$ $\forall k.$ {A tighter result that shows that $\|T(\vu_k) - \vu_k\|$ decreases at rate that is at least as good as $1/\sqrt{k}$ was obtained by~\cite{kohlenbach2011quantitative}.   The results of~\cite{leustean2007rates} and~\cite{kohlenbach2011quantitative} apply to general normed spaces. The work of~\cite{kohlenbach2011quantitative} also provided an explicit rate of metastability that characterizes the convergence of the sequence of iterates $\{\vu_k\}$ in Hilbert spaces.} 

More recently, \cite{lieder2019convergence} proved that under the standard assumption that $T$ has a fixed point $\vu^*$ and for the step size $\lambda_k = \frac{1}{k+1},$  Halpern iteration converges to a fixed point as $\|T(\vu_k) - \vu_k\| = \frac{2\|\vu_0 - \vu^*\|}{k+1}.$ A similar result but for an alternative algorithm was recently obtained by~\cite{kim2019accelerated}. {These two results (as well as all the results from this paper) only apply to Hilbert spaces.} Unlike Halpern iteration, the algorithm introduced by \cite{kim2019accelerated} is not known to possess the implicit regularization property discussed earlier in this paper. 
The results of \cite{lieder2019convergence} and \cite{kim2019accelerated} can be used to obtain the same $1/k$ convergence rate for monotone inclusion with a cocoercive operator \emph{but only if the cocoercivity parameter is known}, which is rarely the case in practice. Similarly, those results can also be extended to more general monotone Lipschitz operators \emph{but only if the proximal map (or resolvent) of $F$ can be computed exactly}, an assumption that can rarely be met  (see Section~\ref{sec:prelims} for definitions of cocoercive operators and proximal maps). We also note that the results of \cite{lieder2019convergence} and \cite{kim2019accelerated} were obtained using the performance estimation (PEP) framework of \cite{drori2014performance}. The convergence proofs resulting from the use of PEP are computer-assisted: they are generated as solutions to large semidefinite programs, which typically makes them hard to interpret and generalize. 

Our approach is arguably simpler, as it relies on the use of a potential function, which allows us to remove the assumptions about the knowledge of the problem parameters and availability of exact proximal maps.   
Our main contributions are summarized as follows:
    \paragraph{Results for cocoercive operators.}
    We introduce a new, potential-based, proof of convergence of Halpern iteration that applies to more general step sizes $\lambda_k$ than handled by the analysis of~\cite{lieder2019convergence} (Section~\ref{sec:main-results}). The proof is simple and only requires elementary algebra. Further, the proof is derived for cocoercive operators and leads to a \emph{parameter-free} algorithm for monotone inclusion. 
    We also extend this parameter-free method to the constrained setting using the concept of gradient mapping generalized to monotone operators (Section~\ref{sec:constrained-case}). To the best of our knowledge, this is the first work to obtain the $1/k$ convergence rate with a parameter-free method.
    %
    \paragraph{Results for monotone Lipschitz operators.}
    Up to a logarithmic factor, we obtain the same $1/k$ convergence rate for the parameter-free setting of the more general monotone Lipschitz operators (Section~\ref{sec:mon-Lip}). The best known convergence rate established by previous work for the same setting was of the order $1/\sqrt{k}$~\citep{dang2015convergence,ryu2019ode}. We obtain the improved convergence rate through the use of the Halpern iteration with \emph{inexact} proximal maps that can be implemented efficiently. The idea of coupling inexact proximal maps with another method is similar in spirit to the Catalyst framework~\citep{lin2017catalyst} and other instantiations of the inexact proximal-point method, such as, e.g., in the work of~\cite{davis2019stochastic,asi2019stochastic,lin2018solving}. However, we note that, unlike in the previous work, the coupling used here is with a method (Halpern iteration) whose convergence properties were not well-understood and for which no simple potential-based convergence proof existed prior to our work. 
    %
    \paragraph{Results for strongly monotone Lipschitz operators.}
    We show that a simple restarting-based approach applied to our method  for operators that are only monotone and Lipschitz (described above) leads to a \emph{parameter-free} method for strongly monotone and Lipschitz operators (Section~\ref{sec:strongly-mon}). Under mild assumptions about the problem parameters and up to a poly-logarithmic factor, the resulting algorithm is \emph{iteration-complexity-optimal}. To the best of our knowledge, this is \emph{the first near-optimal parameter-free method} for the setting of strongly monotone Lipschitz operators and any of the associated problems -- monotone inclusion, VIs, or convex-concave min-max optimization.   
    %
    \paragraph{Lower bounds.}
    To certify near-optimality of the analyzed methods, we provide lower bounds that rely on algorithmic reductions between different problem classes and highlight connections between them (Section~\ref{sec:optimality}). The lower bounds are derived by leveraging the recent lower bound of~\cite{Ouyang2019} for approximating the optimality gap in convex-concave min-max optimization. 
%
%
\subsection{Notation and Preliminaries}\label{sec:prelims}

Let $(E, \|\cdot\|)$ be a real $d$-dimensional Hilbert space, with norm $\|\cdot\| = \sqrt{\innp{\cdot, \cdot}},$ where $\innp{\cdot, \cdot}$ denotes the inner product. In particular, one may consider the Euclidean space $(\rr^d, \|\cdot\|_2).$ Definitions that were already introduced at the beginning of the paper easily generalize from $(\rr^d, \|\cdot\|_2)$ to $(E, \|\cdot\|)$, and are not repeated here for space considerations. 
%
\paragraph{Variational Inequalities and  Monotone Operators.} 
Let $\cu \subseteq E$ be closed and convex, and 
let $F: E \rightarrow E$ be an $L$-Lipschitz-continuous operator defined on $\cu.$ Namely, we assume that:
\begin{equation}\label{eq:lipschitzness}
    (\forall \vu, \vv \in \cu):\quad \|F(\vu) - F(\vv)\| \leq L\|\vu - \vv\|. 
\end{equation}
The definition of monotonicity was already provided in Eq.~\eqref{eq:mon-op}, and easily specializes to monotonicity on the set $\cu$ by restricting $\vu, \vv $ to be from $\cu.$ Further, $F$ is said to be: 
\begin{enumerate}
%
%
\item \emph{strongly monotone} (or \emph{coercive}) on $\cu$ with parameter $m$, if:
\begin{equation}\label{eq:str-mon-op}
    (\forall \vu, \vv \in \cu):\quad \innp{F(\vu) - F(\vv), \vu - \vv} \geq \frac{m}{2}\|\vu - \vv\|^2;
\end{equation}
\item \emph{cocoercive} on $\cu$ with parameter $\gamma$, if:
\begin{equation}\label{eq:cocoercive-op}
    (\forall \vu, \vv \in \cu):\quad \innp{F(\vu) - F(\vv), \vu - \vv} \geq \gamma\|F(\vu) - F(\vv)\|^2.
\end{equation}
\end{enumerate}
It is immediate from the definition of cocoercivity that 
    every $\gamma$-cocoercive operator is monotone and $1/\gamma$-Lipschitz. The latter follows by applying the Cauchy-Schwarz inequality to the left-hand side of Eq.~\eqref{eq:cocoercive-op} and then dividing both sides by $\gamma\|F(\vu) - F(\vv)\|$.
    %

Examples of monotone operators include the gradient of a convex function and appropriately modified gradient of a convex-concave function. Namely, if a function $\Phi(\vx, \vy)$ is convex in $\vx$ and concave in $\vy,$ then $F([\subalign{\vx\\ \vy}]) = [\subalign{\nabla_\vx \Phi(\vx, \vy)\\ -\nabla_{\vy} \Phi(\vx, \vy)}]$ is monotone. 

The Stampacchia Variational Inequality (SVI) problem consists in finding $\vu^* \in \cu$ such that:
\begin{equation}\label{eq:SVI}\tag{SVI}
    (\forall \vu \in \cu): \quad \innp{F(\vu^*), \vu - \vu^*} \geq 0.
\end{equation}
In this case, $\vu^*$ is also referred to as a \emph{strong} solution to the variational inequality (VI) corresponding to $F$ and $\cu$. 
The Minty Variational Inequality (MVI) problem consists in finding $\vu^*$ such that:
\begin{equation}\label{eq:MVI}\tag{MVI}
    (\forall \vu \in \cu): \quad \innp{F(\vu), \vu^* - \vu} \leq 0,
\end{equation}
in which case $\vu^*$ is referred to as a \emph{weak} solution to the variational inequality corresponding to $F$ and $\cu$.  In general, if $F$ is continuous, then the solutions to~\eqref{eq:MVI} are a subset of the solutions to~\eqref{eq:SVI}. 
If we assume that $F$ is monotone, then~\eqref{eq:mon-op} implies that every solution to~\eqref{eq:SVI} is also a solution to \eqref{eq:MVI}, and thus the two solution sets are equivalent.  
The solution set to monotone inclusion is the same as the solution set to~\eqref{eq:SVI}. 

Approximate versions of variational inequality problems \eqref{eq:SVI} and \eqref{eq:MVI} are defined as follows: Given $\epsilon > 0,$ find an $\epsilon$-approximate solution $\vu^*_\epsilon \in \cu,$ which is a solution that satisfies:
\begin{align}
    &(\forall \vu \in \cu): \quad \innp{F(\vu^*_\epsilon), \vu_\epsilon^* - \vu} \leq \epsilon, \;\text{ or}\notag\\
    &(\forall \vu \in \cu): \quad \innp{F(\vu), \vu_\epsilon^* - \vu} \leq \epsilon,\; \text{respectively.}\notag
\end{align}
Clearly, when $F$ is monotone, an $\epsilon$-approximate solution to~\eqref{eq:SVI} is also an $\epsilon$-approximate solution to~\eqref{eq:MVI}; the reverse does \emph{not} hold in general.

Similarly, $\epsilon$-approximate monotone inclusion can be defined as fidning $\vu^*_\epsilon$ that satisfies:
\begin{equation}\label{eq:approx-mon-incl}
    \zeros \in F(\vu^*_\epsilon) + \partial I_{\cu}(\vu^*_\epsilon) + \mathcal{B}(\epsilon),
\end{equation}
where $\mathcal{B}(\epsilon)$ is the ball w.r.t.~$\|\cdot\|$, centered at $\zeros$ and of radius $\epsilon.$ We will sometimes write Eq.~\eqref{eq:approx-mon-incl} in the equivalent form $-F(\vu^*_\epsilon) \in \partial I_{\cu}(\vu^*_\epsilon) + \mathcal{B}(\epsilon).$ The following fact is immediate from Eq.~\eqref{eq:approx-mon-incl}.
\begin{fact}\label{fact:approx-mon-incl}
Given $F$ and $\cu,$ let $\vu^*_\epsilon$ satisfy Eq.~\eqref{eq:approx-mon-incl}. Then:
$$  
    (\forall \vu \in \{\cu \cap \mathcal{B}_{\vu^*_\epsilon}\}): \quad \innp{F(\vu^*_\epsilon), \vu_\epsilon^* - \vu} \leq \epsilon,
$$
where $\mathcal{B}_{\vu^*_\epsilon}$ denotes the unit ball w.r.t.~$\|\cdot\|,$ centered at $\mathcal{B}_{\vu^*_\epsilon}.$

Further, if the diameter of $\cu$, $D = \sup_{\vu, \vv \in \cu}\|\vu - \vv\|$, is bounded, then:
$$  
    (\forall \vu \in \cu ): \quad \innp{F(\vu^*_\epsilon), \vu_\epsilon^* - \vu} \leq \epsilon D.
$$
\end{fact}
Thus, when the diameter $D$ is bounded, any $\frac{\epsilon}{D}$-approximate solution to monotone inclusion is an $\epsilon$-approximate solution to~\eqref{eq:SVI} (and thus also to~\eqref{eq:MVI}); the converse does \emph{not} hold in general. Recall that when $D$ is unbounded, neither~\eqref{eq:SVI} nor~\eqref{eq:MVI} can be approximated. 

We assume throughout the paper that a solution to monotone inclusion~\eqref{eq:monotone-inclusion} exists. This assumption implies that solutions to both~\eqref{eq:SVI} and~\eqref{eq:MVI} exist as well. Existence of solutions follows from standard results and is guaranteed whenever e.g., $\cu$ is compact, or, if there exists a compact set $\cu'$ such that $\mathrm{Id} - \frac{1}{L}F$ maps $\cu'$ to itself~\citep{facchinei2007finite}. 
%
%
\paragraph{Nonexpansive Maps.}

Let $T:E\to E$. We say that $T$ is \emph{nonexpansive} on $\cu \subseteq E$, if $\forall \vu, \vv \in \cu:$
    $$
        \|T(\vu) - T(\vv)\| \leq \|\vu - \vv\|.
    $$
 Nonexpansive maps are closely related to cocoercive operators, and here we summarize some of the basic properties that are used in our analysis. More information can be found in, e.g., the book by~\cite{bauschke2011convex}. 
 \begin{fact}\label{fact:nonexp-cocoerc-equiv}
 $T$ is nonexpansive if and only if $\mathrm{Id} - T$ is $\frac{1}{2}$-cocoercive, where $\mathrm{Id}$ is the identity map.
 \end{fact}
$T$ is said to be \emph{firmly nonexpansive} or \emph{averaged}, if $\forall \vu, \vv \in \cu:$
$$
    \|T(\vu) - T(\vv)\|^2 + \|(\mathrm{Id} - T)\vu - (\mathrm{Id} - T)\vv\|^2 \leq \|\vu - \vv\|^2.
$$
Useful properties of firmly nonexpansive maps are summarized in the following fact. 
\begin{fact}\label{fact:firmly-ne}
For any firmly nonexpansive operator $T,$ $\mathrm{Id} - T$ is also firmly non-expansive, and, moreover, both $T$ and $\mathrm{Id} - T$ are 1-cocoercive. 
\end{fact}
%
%
\section{Halpern Iteration for Monotone Inclusion and Variational Inequalities}\label{sec:main-results}
Halpern iteration is typically stated for nonexpansive maps $T$ as in~\eqref{eq:halpern-or}. 
%
Because our interest is in cocoercive operators $F$ with the unknown parameter $1/L,$ we  instead work with the following version of the Halpern iteration:
\begin{equation}\label{eq:halpern}\tag{H}
    \vu_{k+1} = \lambda_{k+1}\vu_0 + (1-\lambda_{k+1})\Big(\vu_k - \frac{2}{L_{k+1}}F(\vu_k)\Big),
\end{equation}
where $L_k \in (0, \infty),\, \forall k.$ 
If $L$ was known, we could simply set $L_{k+1} = L,$ in which case~\eqref{eq:halpern} would be equivalent to the standard Halpern iteration, due to Fact~\ref{fact:nonexp-cocoerc-equiv}. We assume throughout that $\lambda_1 = \frac{1}{2}.$

We start with the assumption that the setting is unconstrained: $\cu \equiv E.$ We will see in Section~\ref{sec:constrained-case} how the result can be extended to the constrained case. Section~\ref{sec:mon-Lip} will consider the case of operators that are monotone and Lipschitz, while Section~\ref{sec:strongly-mon} will deal with the strongly monotone and Lipschitz case. Some of the proofs are omitted and are instead provided in Appendix~\ref{appx:omitted-proofs}. 

To analyze the convergence of~\eqref{eq:halpern} for the appropriate choices of sequences $\{\lambda_i\}_{i\geq 1}$ and $\{L_i\}_{i\geq 1},$ we  make use of the following potential function:
\begin{equation}\label{eq:potential}
    \cc_k = \frac{1}{L_k}\|F(\vu_k)\|^2 -  \frac{\lambda_k}{1-\lambda_k}\innp{F(\vu_k), \vu_0 - \vu_k}.
\end{equation}

Let us first show that if $A_k\cc_k$ is non-increasing with $k$ for an appropriately chosen sequence of positive numbers $\{A_k\}_{k \geq 1},$  then we can deduce a property that, under suitable conditions on $\{\lambda_i\}_{i\geq 1}$ and $\{ L_i\}_{i\geq 1},$ implies a convergence rate for~\eqref{eq:halpern}.
\begin{lemma}\label{lemma:c-k-non-inc}
    Let $\cc_k$ be defined as in Eq.~\eqref{eq:potential} and let $\vu^*$ be the solution to~\eqref{eq:monotone-inclusion} that minimizes $\|\vu_0 - \vu^*\|$. 
    Assume further that $\innp{F(\vu_1) - F(\vu_0), \vu_1 - \vu_0} \geq \frac{1}{L_1}\|F(\vu_1) - F(\vu_0)\|^2.$ If $A_{k+1}\cc_{k+1} \leq A_k\cc_k,$ $\forall k \geq 1,$ where $\{A_i\}_{i\geq 1}$ is a sequence of positive numbers that satisfies $A_1 = 1$, then:
    $$
       (\forall k \geq 1):\quad \|F(\vu_k)\| \leq L_{k}\frac{\lambda_k}{1-\lambda_k}{\|\vu_0 - \vu^*\|}.
    $$
\end{lemma}
%

Using Lemma~\ref{lemma:c-k-non-inc}, our goal is now to show that we can choose $L_k = O(L)$ and $\lambda_k = O(\frac{1}{k}),$ which in turn would imply the desired $1/k$ convergence rate: $\|F(\vu_k)\| = O(\frac{L\|\vu_0 - \vu^*\|}{k}).$ 
The following lemma provides sufficient conditions for $\{A_i\}_{i \geq 1},$ $\{\lambda_i\}_{i\geq 1}$, and $\{ L_i\}_{i\geq 1}$ to ensure that $A_{k+1}\cc_{k+1} \leq A_k\cc_k,$ $\forall k \geq 1,$ so that Lemma~\ref{lemma:c-k-non-inc} applies. 
\begin{lemma}\label{lemma:pot-dec}
    Let $\cc_k$ be defined as in Eq.~\eqref{eq:potential}. Let $\{A_i\}_{i \geq 1}$ be defined recursively as $A_1 = 1$ and $A_{k+1} = A_k \frac{\lambda_k}{(1-\lambda_k)\lambda_{k+1}}$ for $k \geq 1.$ Assume that $\{\lambda_i\}_{i\geq 1}$ is chosen so that $\lambda_1 = \frac{1}{2}$ and for $k \geq 1:$ $\frac{\lambda_{k+1}}{1-2\lambda_{k+1}} \geq  \frac{\lambda_k L_k}{(1-\lambda_k)L_{k+1}}$. Finally, assume that $L_k \in (0, \infty)$ and $\innp{F(\vu_k) - F(\vu_{k-1}), \vu_k - \vu_{k-1}}\geq \frac{1}{L_k}\|F(\vu_k) - F(\vu_{k-1})\|^2$, $\forall k.$ Then, 
    $$(\forall k \geq 1): \quad A_{k+1}\cc_{k+1} \leq A_k\cc_k.$$ 
\end{lemma}

Observe first the following. If we knew $L$ and set $L_k = L,$ $\lambda_k = \frac{1}{k+1},$ and $A_k = k(k+1)/2,$ then all of the conditions from Lemma~\ref{lemma:pot-dec} would be satisfied, and Lemma~\ref{lemma:c-k-non-inc} would then imply $\|F(\vu_k)\| \leq \frac{L\|\vu_0 - \vu^*\|}{k},$ which recovers the result of~\cite{lieder2019convergence}. The choice $\lambda_k = \frac{1}{k+1}$ is also the tightest possible that satisfies the conditions  Lemma~\ref{lemma:pot-dec} -- the inequality relating $\lambda_{k+1}$ and $\lambda_k$ is satisfied with equality. This result is in line with the numerical observations made by~\cite{lieder2019convergence}, who observed that the convergence of Halpern iteration is  fastest for $\lambda_k = \frac{1}{k+1}$. 

To construct a parameter-free method, we use that $F$ is $L$-cocoercive; namely, that there exists a constant $L < \infty$ such that $F$ satisfies Eq.~\eqref{eq:cocoercive-op} with $\gamma = 1/L$. The idea is to start to with a ``guess'' of $L$ (e.g., $L_0 = 1$) and double the guess $L_k$ as long as $\innp{F(\vu_k) - F(\vu_{k-1}), \vu_k - \vu_{k-1}} < \frac{1}{L_k}\|F(\vu_k) - F(\vu_{k-1})\|^2.$ The total number of times that the guess can be doubled is bounded above by $\max\{0, \log_2(2L/L_0)\}.$ Parameter $\lambda_k$ is simply chosen to satisfy the condition from Lemma~\ref{lemma:pot-dec}. The algorithm pseudocode is stated in Algorithm~\ref{algo:halpern-basic} for a given accuracy specified at the input. 
\begin{algorithm}
\caption{Parameter-Free Halpern -- Cocoercive Case}
\label{algo:halpern-basic}
\textbf{Input:} $L_0 > 0$, $\epsilon > 0$, $\vu_0$. If not provided at the input, set $L_0 = 1.$\;

$\lambda_1 = \frac{1}{2}$, $k = 0$\;

\While{$\|F(\vu_k)\| > \epsilon$}
{
$k = k+1$\;

$ L_k = L_{k-1}$\;

$p_k = \frac{L_{k-1}}{L_k}\frac{\lambda_{k-1}}{1-\lambda_{k-1}}\,$, $\lambda_k = \frac{p_k}{1 + 2p_k}$\;

$\vu_k = \lambda_k \vu_0 + (1-\lambda_k)(\vu_{k-1} - 2 F(\vu_{k-1})/ L_k)$\;

\While{$\innp{F(\vu_k) - F(\vu_{k-1}), \vu_k - \vu_{k-1}} < \frac{1}{L_k}\|F(\vu_k) - F(\vu_{k-1})\|^2$}
{
$L_k = 2\cdot L_k$\;

$p_k = \frac{L_{k-1}}{L_k}\frac{\lambda_{k-1}}{1-\lambda_{k-1}}\,$, $\lambda_k = \frac{p_k}{1 + 2p_k}$\;

$\vu_k = \lambda_k \vu_0 + (1-\lambda_k)(\vu_{k-1} - 2 F(\vu_{k-1})/ L_k)$\;
}
}
\Return 
$\vu_k$
\end{algorithm}

We now prove the first of our main results. Note that the total number of arithmetic operations in  Algorithm~\ref{algo:halpern-basic} is of the order of the number of oracle queries to $F$ multiplied by the complexity of evaluating $F$ at a point. The same will be true for all the algorithms stated in this paper, except that the complexity of evaluating $F$ may be replaced by the complexity of  projections onto $\cu$.   
\begin{theorem}\label{thm:halpern-cocoercive}
Given $\vu_0\in \cu$ and an operator $F$ that is $\frac{1}{L}$-cocoercive on $E,$ Algorithm~\ref{algo:halpern-basic} returns a point $\vu_k$ such that $\|F(\vu_k)\|\leq \epsilon$ after at most $\frac{\max\{2L, L_0\}\|\vu_0 - \vu^*\|}{\epsilon} + \max\{0, \log_2(2L/L_0)\}$ oracle queries to $F$.
\end{theorem}
\begin{proof}
As $F$ is $\frac{1}{L}$-cocoercive, $L_k \leq \max\{2L, L_0\}$ and the total number of times that the algorithm enters the inner while loop is at most $\max\{0, \log_2(2L/L_0)\}.$ The parameters satisfy the assumptions of Lemmas~\ref{lemma:c-k-non-inc} and~\ref{lemma:pot-dec}, and, thus, $\|F(\vu_k)\|\leq L_k \frac{\lambda_k}{1-\lambda_k}\|\vu_0 - \vu^*\|.$ Hence, we only need to show that $\lambda_k$ decreases sufficiently fast with $k.$ 
As $L_k$ can only be increased in any iteration, we have that 
\begin{align*}{\lambda_{k+1}} \leq \frac{\frac{\lambda_k}{1-\lambda_k}}{1+ 2\frac{\lambda_k}{1-\lambda_k}} = \frac{\lambda_k}{1+\lambda_k} \leq \frac{\lambda_{k-1}}{1 + 2\lambda_{k-1}} \leq \dots \leq  \frac{\lambda_1}{1+k\lambda_1} = \frac{1}{k + 2}.
\end{align*}
Hence, the total number of outer iterations is at most $\frac{\max\{2L, L_0\}\|\vu_0 - \vu^*\|}{\epsilon}$. Combining with the maximum total number of inner iterations from the beginning of the proof, the result follows.
\end{proof}
\subsection{Constrained Setups with Cocoercive Operators}\label{sec:constrained-case}
Assume now that $\cu \subseteq E.$ We will make use of a counterpart to \emph{gradient mapping}~\cite[Chapter 2]{nesterov2018lectures} that we refer to as the \emph{operator mapping}, defined as:
\begin{equation}\label{eq:op-mapping}
    G_{\eta}(\vu) = \eta\Big(\vu - \Pi_{\cu}\Big(\vu - \frac{1}{\eta}F(\vu)\Big)\Big),
\end{equation}
where $\Pi_{\cu}\big(\vu - \frac{1}{\eta}F(\vu)\big)$ is the projection operator, namely:
\begin{align*}
    \Pi_{\cu}\Big(\vu - \frac{1}{\eta}F(\vu)\Big) = \argmin_{\vv\in \cu}\Big\{\frac{1}{2}\|\vv - \vu + F(\vu)/\eta\|^2\Big\} = \argmin_{\vv\in \cu}\Big\{\innp{F(\vu), \vv} + \frac{\eta}{2}\|\vv - \vu\|^2\Big\}.
\end{align*}
Operator mapping generalizes a cocoercive operator to the constrained case: when $\cu \equiv E,$ $G_\eta \equiv F.$  

It is a well-known fact that the projection operator is firmly-nonexpansive~\cite[Proposition 4.16]{bauschke2011convex}. 
Thus, Fact~\ref{fact:firmly-ne} can be used to show that, if $F$ is $\frac{1}{L}$-cocoercive and $\eta \geq L,$ then $G_\eta$ is $\frac{1}{2\eta}$-cocoercive. This is shown in the following (simple) proposition.
\begin{proposition}\label{prop:op-map-cocoercive}
Let $F$ be an $\frac{1}{L}$-cocoercive operator and let $G_\eta$ be defined as in Eq.~\eqref{eq:mon-op}, where $\eta \geq L.$ Then $G_\eta$ is $\frac{1}{2\eta}$-cocoercive. 
\end{proposition}
%
%

As $G_\eta$ is $\frac{1}{2\eta}$-cocoercive, applying results from the beginning of the section to $G_\eta$, it is now immediate that Algorithm~\ref{algo:halpern-constr} (provided for completeness) produces $\vu_k$ with $\|G_{L_k}(\vu_k)\| \leq \epsilon$ after at most $\frac{\max\{4L, L_0\}\|\vu_0 - \vu^*\|}{\epsilon} + \max\{0, \log_2(4L/L_0)\}$ oracle queries to $F$ (as each computation of $G_\eta$ requires one oracle query to $F$). 
\begin{algorithm}
\caption{Parameter-Free Halpern -- Cocoercive and Constrained Case}
\label{algo:halpern-constr}
\textbf{Input:} $L_0 > 0$, $\epsilon > 0$, $\vu_0\in \cu$. If not provided at the input, set $L_0 = 1.$\;

$\lambda_1 = \frac{1}{2}$, $k = 0$\;

$\vub_0 = \Pi_{\cu}(\vu_0 - F(\vu_0)/L_0)\,$, $\bar{L}_0 = \frac{\|F(\vub_0)- F(\vu_0)\|}{\|\vub_0 - \vu_0\|}$\;

\While{$\|G_{L_k}(\vu_k)\| > \epsilon/(1+ \bar{L}_k/L_k)$}
{
\nl$k = k+1$\;

\nl$ L_k = L_{k-1}$\;

\nl$p_k = \frac{L_{k-1}}{L_k}\frac{\lambda_{k-1}}{1-\lambda_{k-1}}\,$, $\lambda_k = \frac{p_k}{1 + 2p_k}$\;

\nl$\vu_k = \lambda_k \vu_0 + (1-\lambda_k)\vub_{k-1}$\;

\While{$\innp{G_{L_k}(\vu_k) - G_{L_k}(\vu_{k-1}), \vu_k - \vu_{k-1}} < \frac{1}{2L_k}\|G_{L_k}(\vu_k) - G_{L_k}(\vu_{k-1})\|^2$}
{
\nl$L_k = 2\cdot L_k$\;

\nl$p_k = \frac{L_{k-1}}{L_k}\frac{\lambda_{k-1}}{1-\lambda_{k-1}}\,$, $\lambda_k = \frac{p_k}{1 + 2p_k}$\;

\nl$\vu_k = \lambda_k \vu_0 + (1-\lambda_k)(\vu_{k-1} - G_{L_k}(\vu_{k-1})/ L_k)$\;
}

\nl$\vub_k = \Pi_{\cu}(\vu_k - F(\vu_k)/L_k),\,$ $\bar{L}_k = \frac{\|F(\vub_k)- F(\vu_k)\|}{\|\vub_k - \vu_k\|}\,$, $L_k = \max\{L_k,\, \bar{L}_k\}$ \label{line:Lk-etak}\;
}
\Return $\vub_k$, $\vu_k$
\end{algorithm}
To complete this subsection, it remains to show that $G_\eta$ is a good surrogate for approximating~\eqref{eq:monotone-inclusion} (and \eqref{eq:SVI}). This is indeed the case and it follows as a suitable generalization of Lemma 3 from~\cite{ghadimi2016accelerated}, which is provided here for completeness.
\begin{lemma}\label{lemma:grad-mapping-approx}
Let $G_\eta$ be defined as in Eq.~\eqref{eq:op-mapping}. Denote $\bar{\vu} = \Pi_{\cu}(\vu - F(\vu)/\eta),$ so that $G_\eta(\vu) = \eta(\vu - \bar{\vu}).$ If, for some $\vu \in \cu,$ $\|G_\eta(\vu)\|\leq \epsilon,$ then 
$$
F(\vub) \in - \partial I_{\cu}(\vub) + {\cal B}((1+ L_{\mathrm{loc}}/\eta)\epsilon),$$ 
where 
$L_{\mathrm{loc}} = \frac{\|F(\vub) - F(\vu)\|}{\|\vub - \vu\|} \leq L$. 
\end{lemma}
\begin{proof}
 As, by definition, $\bar{\vu} = \argmin_{\vv\in \cu}\big\{\innp{F(\vu), \vv} + \frac{\eta}{2}\|\vv - \vu\|^2\big\},$ by first-order optimality of $\vub,$ we have: 
$
    \zeros \in F(\vu) + \eta(\vub - \vu) + \partial I_{\cu}(\vub). 
$ 
Equivalently: $-F(\vub) \in F(\vu) - F(\vub) - G_\eta(\vu) + \partial I_\cu(\vub).$ The rest of the proof follows simply by using $\|G_{\eta}\|\leq \epsilon$ and $\|F(\vu) - F(\vub)\| = L_{\mathrm{loc}} \|\vu - \vub\| = \frac{L_{\mathrm{loc}}}{\eta}\|G_{\eta}(\vu)\| \leq \frac{L_{\mathrm{loc}}}{\eta} \epsilon.$
\end{proof}
Lemma~\ref{lemma:grad-mapping-approx} implies that when the operator mapping is small in norm $\|\cdot\|,$ then $\vub = \Pi_{\cu}(\vu - F(\vu)/\eta)$ is an approximate solution to~\eqref{eq:monotone-inclusion} corresponding to $F$ on $\cu.$ 
%
%
We can now formally bound the number of oracle queries to $F$ needed to approximate \eqref{eq:monotone-inclusion} and~\eqref{eq:SVI}.
\begin{theorem}\label{thm:halpern-cocoerc-constrained}
Given $\vu_0\in \cu$ and a $\frac{1}{L}$-cocoercive operator $F$,  Algorithm~\ref{algo:halpern-constr} returns $\vub_k \in \cu$ such that
\begin{enumerate}
    \item $\|G_{L_k}(\vub_k)\|\leq \frac{\epsilon}{2}$, $\max_{\vv \in \{\cu \cap {\cal B}_{\vub_k} \}}\innp{F(\vub_k), \vub_k - \vv} \leq \epsilon$ after at most $$\frac{4\max\{4L, L_0\}\|\vu_0 - \vu^*\|}{\epsilon} + 2\max\{0, \log_2(4L/L_0)\}$$ oracle queries to $F;$
    \item $\max_{\vv \in \cu}\innp{F(\vub), \vub - \vv} \leq \epsilon$ after at most  $$\frac{4\max\{4L, L_0\}\|\vu_0 - \vu^*\|D}{\epsilon} + 2\max\{0, \log_2(4L/L_0)\}$$  oracle queries to $F.$
\end{enumerate}
Further, every point $\vu_k$ that Algorithm~\ref{algo:halpern-constr} constructs is from the feasible set: $\vu_k \in \cu,$ $\forall k \geq 0$, and a simple modification to the algorithm takes at most $\frac{\max\{4L, L_0\}\|\vu_0 - \vu^*\|}{\epsilon} + \max\{0, \log_2(4L/L_0)\}$ oracle queries to $F$ to construct a point such that $\|G_{L_k}(\vu_k)\|\leq \epsilon$.
\end{theorem}
\begin{proof}
 By the definition of $G_\eta,$ if $\vu_0 \in \cu,$ then $\vu_k \in \cu,$ for all $k.$ This follows simply as:
\begin{align*}
    \vu_{k+1} &= \lambda_{k+1}\vu_0 + (1-\lambda_{k+1})\Big(\vu_k - \frac{1}{L_{k+1}}G_{L_{k+1}}(\vu_k)\Big)\\
                & = \lambda_{k+1}\vu_0 + (1-\lambda_{k+1})\Pi_{\cu}(\vu_k - F(\vu_k)/L_{k+1}).
\end{align*}
Observe that, due to Line~\ref{line:Lk-etak} of Algorithm~\ref{algo:halpern-constr}, $L_k \geq \bar{L}_k.$ The rest of the proof follows using Lemma~\ref{lemma:grad-mapping-approx}, Fact~\ref{fact:approx-mon-incl}, and the same reasoning as in the proof of Theorem~\ref{thm:halpern-cocoercive}. Observe that if the goal is to only output a point $\vu_k$ such that $\|G_{L_k}(\vu_k)\| \leq \epsilon$, then computing $\vub_k$ and $F(\vub_k)$ is not needed, and the algorithm can instead use $\|G_{L_k}(\vu_k)\| > \epsilon$ as the exit condition in the outer while loop. 
\end{proof}
%
\subsection{Setups with non-Cocoercive Lipschitz Operators}\label{sec:mon-Lip}
We now consider the case in which $F$ is not cocoercive, but only monotone and $L$-Lipschitz. To obtain the desired convergence result, we  make use of the resolvent operator, defined as $J_{F+\partial I_\cu} = (\mathrm{Id} + F + \partial I_\cu)^{-1}.$ A useful property of the resolvent is that it is firmly-nonexpansive~\cite[and references therein]{ryu2016primer}, which, due to Fact~\ref{fact:firmly-ne}, implies that $P = \mathrm{Id} - J_{F+ \partial I_\cu}$ is $\frac{1}{2}$-cocoercive. 

Finding a point $\vu \in \cu$ such that $\|P(\vu)\|\leq \epsilon$ is sufficient for approximating monotone inclusion (and~\eqref{eq:SVI}). This is shown in the following simple proposition, provided here for completeness.
\begin{proposition}\label{prop:P-as-an-approx}
Let $P = \mathrm{Id} - J_{F+ \partial I_\cu}$. If $\|P(\vu)\|\leq \epsilon$, then $\vub = \vu - P(\vu) = J_{F+ \partial I_\cu}(\vu)$ satisfies 
$$F(\vub) \in - \partial I_\cu(\vub) + \mathcal{B}(\epsilon).$$
\end{proposition}
\begin{proof}
 By the definition of $P$ and $J_{F+ \partial I_\cu}$, $\vu - P(\vu) = (\mathrm{Id} + F + \partial I_\cu)^{-1}(\vu).$ Equivalently:
 $$
    \vu - P(\vu) + F(\vu - P(\vu)) + \partial I_\cu(\vu - P(\vu)) \ni \vu.
 $$
 As $\|P(\vu)\|\leq \epsilon,$ the result follows.
\end{proof}

If we could compute the resolvent exactly, it would suffice to directly apply the result of~\cite{lieder2019convergence}. However, excluding very special cases, computing the exact resolvent efficiently is generally not possible. However, since $F$ is Lipschitz, the resolvent $J_{F+ \partial I_\cu}$ can be \emph{approximated} efficiently. This is because it corresponds to solving a VI defined on a closed convex set $\cu$ with the operator $F + \mathrm{Id}$ that is $1$-strongly monotone and $(L+1)$-Lipschitz. Thus, it can be computed by solving a strongly monotone and Lipschitz VI, for which one can use the results of  e.g.,~\cite{nesterov2011solving,mokhtari2019unified,gidel2018variational} if $L$ is known, or~\cite{stonyakin2018generalized}, if $L$ is not known. For completeness, we provide a simple  modification to the Extragradient algorithm of~\cite{korpelevich1977extragradient} in Algorithm~\ref{algo:eg+}  (Appendix~\ref{appx:omitted-proofs}), for which we prove that it attains the optimal convergence rate without the knowledge of $L$. The convergence result is summarized in the following lemma, whose proof is provided in Appendix~\ref{appx:omitted-proofs}. 
\begin{lemma}
\label{lemma:approx-resolvent}
Let $\vub_k^* = J_{F+I_{\cu}}(\vu_k),$  %
where $\vu_k\in \cu$ and $F$ is $L$-Lipschitz. Then, there exists a parameter-free algorithm that queries $F$ at most $O((L+1)\log(\frac{L\|\vu_k - \vub_k^*\|}{\epsilon}))$ times and outputs a point $\vub_k$ such that $\|\vub_k - \vub^*_k\|\leq \epsilon.$ 
\end{lemma}
%

%
%
To obtain the desired result, we need to prove the convergence of a Halpern iteration with inexact evaluations of the cocoercive operator $P$. Note that here we do know the cocoercivity parameter of $P$ -- it is equal to $1/2$. 
%
%
The resulting inexact version of Halpern's iteration for $P$ is:
\begin{equation}\label{eq:inexact-halpern}
\begin{aligned}
    \vu_{k+1} &= \lambda_{k+1}\vu_0 + (1-\lambda_{k+1})(\vu_k - \Tilde{P}(\vu_k))\\
    & = \lambda_{k+1}\vu_0 + (1-\lambda_{k+1})\Tilde{J}_{F+\partial I_{\cu}}(\vu_k),
\end{aligned}
\end{equation}
where $\Tilde{P}(\vu_k) - P(\vu_k)  =  {J}_{F+\partial I_{\cu}}(\vu_k) - \Tilde{J}_{F+\partial I_{\cu}}(\vu_k) = \ve_k$ is the error. 

To analyze the convergence of~\eqref{eq:inexact-halpern}, we again use the potential function $\cc_k$ from Eq.~\eqref{eq:potential}, with $P$ as the operator.  For simplicity of exposition, we take the best choice of $\lambda_i = \frac{1}{i+1}$ that can be obtained from Lemma~\ref{lemma:c-k-non-inc} for $ L_i = L = 2,$ $\forall i.$ 
The key result for this setting is provided in the following lemma, whose proof is deferred to the appendix. 
\begin{lemma}\label{lemma:inexact-halpern}
 Let $\cc_k$ be defined as in Eq.~\eqref{eq:potential} with $P$ as the $\frac{1}{2}$-cocoercive operator, and let $L_k = 2,$ $\lambda_k = \frac{1}{k+1},$ and $A_k = \frac{k(k+1)}{2},$ $\forall k \geq 1$. If the iterates $\vu_k$ evolve according to~\eqref{eq:inexact-halpern} for an arbitrary initial point $\vu_0 \in \cu,$ then:  
    $$(\forall k \geq 1): \quad A_{k+1}\cc_{k+1} \leq A_k\cc_k + A_{k+1}\innp{\ve_k, (1-\lambda_{k+1})P(\vu_k) - P(\vu_{k+1})}.$$ 
Further, if, $\forall k \geq 1,$ $\|\ve_{k-1}\|\leq \frac{\epsilon}{4 k(k+1)},$ then $\|P(\vu_K)\|\leq \epsilon$ after at most $K = \frac{4\|\vu_0 - \vu^*\|}{\epsilon}$ iterations.
\end{lemma}
We are now ready to state the algorithm and prove the main theorem for this subsection.
\begin{algorithm}
\caption{Parameter-Free Halpern -- Monotone and Lipschitz Case}\label{algo:mon-Lip}
\textbf{Input:} $\epsilon > 0$, $\vu_0\in \cu$ \;

$k = 0$, $\epsilon_0 = \frac{\epsilon}{8}$\;

$\vub_0 = \Tilde{J}_{F + \partial I_{\cu}}(\vu_0),$
where $\|\Tilde{J}_{F + \partial I_{\cu}}(\vu_0) - {J}_{F + \partial I_{\cu}}(\vu_0)\| \leq \epsilon_0$\;

$\Tilde{P}(\vu_0) = \vu_0 - \vub_0$\;

\While{$\|\Tilde{P}(\vu_k)\|>\frac{3\epsilon}{4}$}
{
$k = k+1$, $\lambda_k = \frac{1}{k+1}$, $\epsilon_k = \frac{\epsilon}{8(k+1)(k+2)}$\;

$\vu_k = \lambda_k \vu_0 + (1-\lambda_k)\vub_{k-1}$\;

$\vub_k = \Tilde{J}_{F + \partial I_{\cu}}(\vu_k),$ where $\|\Tilde{J}_{F + \partial I_{\cu}}(\vu_k) - {J}_{F + \partial I_{\cu}}(\vu_k)\| \leq \epsilon_k$\;

$\Tilde{P}(\vu_k) = \vu_k - \vub_k$\;
}
\Return $\vub_k$, $\vu_k$
\end{algorithm}
\begin{theorem}\label{thm:mon-op}
Let $F$ be a monotone and $L$-Lipschitz operator and let $\vu_0 \in \cu$ be an arbitrary initial point. For any $\epsilon > 0,$ Algorithm~\ref{algo:mon-Lip} outputs a point with $\|P(\vu_k)\| \leq \epsilon$ after at most $\frac{8\|\vu^* - \vu_0\|}{\epsilon}$ iterations, where each iteration can be implemented with $O((L+1) \log(\frac{(L+1)\|\vu_0 - \vu^*\|}{\epsilon})$ oracle queries to $F.$ Hence, the total number of oracle queries to $F$ is:
$
    O\big(\frac{(L+1)\|\vu_0 - \vu^*\|}{\epsilon}\log\big(\frac{(L+1)\|\vu_0 - \vu^*\|}{\epsilon}\big)\big).
$
\end{theorem}
\begin{proof}
Recall that $\Tilde{P}(\vu_k) - P(\vu_k) = \ve_k$ and $\|\ve_k\| = \epsilon_k = \frac{\epsilon}{8(k+1)(k+2)} < \frac{\epsilon}{4}.$ Hence, as Algorithm~\ref{algo:mon-Lip} outputs a point $\vu_k$ with $\|\Tilde{P}(\vu_k)\| \leq \frac{3\epsilon}{4},$ by the triangle inequality, $\|P(\vu_k)\| \leq \epsilon.$

To bound the number of iterations until $\|\Tilde{P}(\vu_k)\| \leq \frac{3\epsilon}{4},$ note that, again by the triangle inequality, if $\|P(\vu_k)\| \leq \epsilon/2$, then $\|\Tilde{P}(\vu_k)\| \leq \frac{3\epsilon}{4}.$ Applying Lemma~\ref{lemma:inexact-halpern}, $\|P(\vu_k)\| \leq \epsilon/2$ after at most $k = \frac{8\|\vu_0 - \vu^*\|}{\epsilon}$ iterations, completing the proof of the first part of the theorem.

For the remaining part, using Lemma~\ref{lemma:approx-resolvent}, $\Tilde{J}_{F + \partial I_{\cu}}(\vu_k)$ can be computed (with target error $\epsilon_k$) in $O((L+1)\log(\frac{(L+1)\|\vu_k - {J}_{F + \partial I_{\cu}}(\vu_k)\|}{\epsilon_k})) = O((L+1)\log(\frac{(L+1)\|P(\vu_k)\|}{\epsilon}))$ iterations, as $O(\log(\frac{1}{\epsilon_k})) = O(\log(\frac{1}{\epsilon}))$ and $P(\vu_k) = \vu_k - {J}_{F + \partial I_{\cu}}(\vu_k),$ by definition. It remains to use that $\|P(\vu_k)\| = O(\|\vu_0 - \vu^*\|)$, which can be deduced from, e.g., Eq.~\eqref{eq:in-halp-err-final} in the proof of Lemma~\ref{lemma:inexact-halpern}.
\end{proof}
Similarly as before, $\|P(\vu_k)\| \leq \epsilon$ implies an $\epsilon$-approximate solution to~\eqref{eq:monotone-inclusion}, by Proposition~\ref{prop:P-as-an-approx}. When the diameter $D$ is bounded, $\|P(\vu_k)\| \leq \frac{\epsilon}{D}$ implies an $\epsilon$-approximate solution to~\eqref{eq:SVI}.
\begin{remark}
In degenerate cases where $L << 1$, instead of using the resolvent of $F + \partial I_{\cu}$, one could use the resolvent of $F/\eta + \partial I_{\cu}$ for $\eta = O(L),$ assuming the order of magnitude of $L$ is known (this is typically a mild assumption). Then, each approximate computation of the resolvent would take $O((L/\eta+1) \log(\frac{(L/\eta+1)\|\vu_0 - \vu^*\|}{\epsilon})$ oracle queries to $F,$ and we would need to require that $\|\Tilde{P}(\vu_k)\| \leq \frac{3\epsilon}{4\eta}$. Thus, the total number of queries to $F$ would be $O((L+\eta)\log(\frac{(L+\eta)\|\vu_0 - \vu^*\|}{\epsilon})).$
\end{remark}
%
%
\subsection{Setups with Strongly Monotone and Lipschitz Operators}\label{sec:strongly-mon}
%
%
We now show that by restarting Algorithm~\ref{algo:mon-Lip}, we can obtain a parameter-free method with near-optimal oracle complexity. To simplify the exposition, we assume w.l.o.g.~that $L = \Omega(1).$ 
\begin{theorem}\label{thm:restarting-Lipschitz}
Given $F$ that is $L$-Lipschitz and $m$-strongly monotone, consider running the following algorithm $\mathcal{A}$, starting with $\vu_0\in \cu$:
\begin{center}
    $(\mathcal{A}):\quad$ At iteration $k,$ invoke Algorithm~\ref{algo:mon-Lip} with error parameter $\epsilon_k = \frac{7}{16}\|\Tilde{P}(\vu_{k-1})\|.$
\end{center}
Then, $\mathcal{A}$ outputs $\vu_k \in \cu$ with $\|P(\vu_k)\|\leq \epsilon$ after at most $1+\log_2(\frac{\|\vu_0 - \vu^*\|}{\epsilon})$ iterations, for any~$\epsilon \in (0, \frac{1}{2}]$. The total number of  queries to $F$ until $\|P(\vu_k)\|\leq \epsilon$ is $O\big( (L + \frac{L}{m})\log(\frac{\|\vu_0 - \vu^*\|}{\epsilon})\log(L + \frac{L}{m})\big).$ 
\end{theorem}
\begin{proof}
The first part  is immediate, as each call to Algorithm~\ref{algo:mon-Lip} ensures, due to Theorem~\ref{thm:mon-op}, that 
$$
\|P(\vu_k)\| \leq \frac{\|7\Tilde{P}(\vu_{k-1})\|}{16} \leq \frac{7\|{P}(\vu_{k-1})\|}{16} + \frac{\epsilon_k}{8} \leq \frac{\|{P}(\vu_{k-1})\|}{2},
$$ 
and $\|P(\vu_0)\| \leq 2\|\vu_0 - \vu^*\|$ as $P$ is 2-Lipschitz (because it is $\frac{1}{2}$-cocoercive) and $P(\vu^*) = \zeros.$

It remains to bound the number of calls to $F$ for each call to Algorithm~\ref{algo:mon-Lip}. Using Theorem~\ref{thm:mon-op} and $\|\Tilde{P}(\vu_k)\| = \Theta(\|P(\vu_k)\|)$, each call to Algorithm~\ref{algo:mon-Lip} takes $O(\frac{L\|\vu_{k-1} - \vu^*\|}{\|P(\vu_{k-1})\|} \log(\frac{L\|\vu_{k-1} - \vu^*\|}{\|P(\vu_{k-1})\|}))$ calls to $F.$ 
Denote $\vub^*_{k-1} = J_{F+\partial I_{\cu}}(\vu_{k-1}) = \vu_{k-1} - P(\vu_{k-1}).$ Using Proposition~\ref{prop:P-as-an-approx}:
$$
    \innp{F(\vub^*_{k-1}), \vub^*_{k-1} - \vu^*}\leq \|P(\vu_{k-1})\|\|\vub^*_{k-1} - \vu^*\|.
$$
On the other hand, as $F$ is $m$-strongly monotone and $\vu^*$ is an~\eqref{eq:MVI} solution, 
$$
    m \|\vub^*_{k-1} - \vu^*\|^2 \leq \innp{F(\vub_{k-1}^*), \vub^*_{k-1} - \vu^*}.
$$ 
Hence: $\|\vub^*_{k-1} - \vu^*\| \leq \frac{1}{m}\|P(\vu_{k-1})\|.$ It remains to use the triangle inequality and $P(\vu_{k-1}) = \vu_{k-1} - \vub^*_{k-1}$ to obtain:
$
    \|\vu_{k-1} - \vu^*\| \leq \big(1 + \frac{1}{m}\big)\|P(\vu_{k-1})\|.
$
\end{proof}
%
%
\section{Lower Bound Reductions}\label{sec:optimality}
In this section, we only state the lower bounds, while more details about the oracle model and the proof are deferred to Appendix~\ref{appx:omitted-proofs}.
\begin{lemma}\label{lemma:lower-bounds}
For any deterministic algorithm working in the operator oracle model and any $L, D >0$, there exists an $L$-Lipschitz-continuous operator $F$ and a closed convex feasible set $\cu$ with diameter $D$ such that:
\begin{enumerate}[label=(\alph*)]
    \item For all $\epsilon>0$ such that $k = \frac{L D^2}{\epsilon} = O(d)$, $\max_{\vu \in \cu} \innp{F(\vu_k), \vu_k - \vu} = \Omega(\epsilon)$;\label{item:lb-SVI}
    \item For all $\epsilon>0$ such that $k = \frac{L D}{\epsilon} = O(d)$, $\max_{\vu \in \{\cu\cap \mathcal{B}_{\vu_k}\}} \innp{F(\vu_k), \vu_k - \vu} = \Omega(\epsilon)$; \label{item:lb-MI-Lip}
    \item  If $F$ is $\frac{1}{L}$-cocoercive, then for all $\epsilon>0$ such that $k = \frac{L D}{\epsilon \log(D/\epsilon)} = O(d)$, it holds that  $$\max_{\vu \in \{\cu\cap \mathcal{B}_{\vu_k}\}} \innp{F(\vu_k), \vu_k - \vu} = \Omega(\epsilon);$$\label{item:MI-lb-coco}
    \item If $F$ is $m$-strongly monotone, then for all $\epsilon>0$ such that $k = \frac{L}{m} = O(d)$, it holds that  $$\max_{\vu \in \{\cu\cap \mathcal{B}_{\vu_k}\}} \innp{F(\vu_k), \vu_k - \vu} = \Omega(\epsilon).$$ \label{item:MI-lb-str-mon}
\end{enumerate}
\end{lemma}

Parts~\ref{item:lb-SVI} and~\ref{item:lb-MI-Lip} of Lemma~\ref{lemma:lower-bounds} certify that Algorithm~\ref{algo:mon-Lip} is optimal up to a logarithmic factor, due to Theorem~\ref{thm:mon-op}. This is true because we can run Algorithm~\ref{algo:mon-Lip} with accuracy $\frac{\epsilon}{D}$ to obtain $\max_{\vu \in \cu} \innp{F(\vu_k), \vu_k - \vu} = O(\epsilon)$ in $k = O(\frac{L D^2}{\epsilon}\log(\frac{LD}{\epsilon}))$ iterations, or with accuracy $\epsilon$ to obtain $\max_{\vu \in \{\cu\cap \mathcal{B}_{\vu_k}\}} \innp{F(\vu_k), \vu_k - \vu} = O(\epsilon)$ in $k = O(\frac{LD}{\epsilon}\log(\frac{LD}{\epsilon}))$ iterations (see Proposition~\ref{prop:P-as-an-approx}). 

Part~\ref{item:MI-lb-coco} of Lemma~\ref{lemma:lower-bounds} certifies that Algorithm~\ref{algo:halpern-constr} is optimal up to a $\log(D/\epsilon)$ factor, due to Theorem~\ref{thm:halpern-cocoerc-constrained}. Part~\ref{item:MI-lb-str-mon} certifies that the restarting algorithm from Theorem~\ref{thm:restarting-Lipschitz} is optimal up to a factor $\log(D/\epsilon)\log(L/m)$ whenever $L = \Omega(L/m).$ Note that $L = \Omega(L/m)$ can be ensured by a proper scaling of the problem instance, as any such scaling would leave the condition number $L/m$ unaffected and would only impact the target error $\epsilon,$ which only appears under a logarithm. 

\section{Conclusion}
We showed that variants of Halpern iteration can be used to obtain near-optimal methods for solving different classes of monotone inclusion problems with Lipschitz operators. The results highlight connections between monotone inclusion, variational inequalities, fixed points of nonexpansive maps, and proximal-point-type algorithms. 
Some interesting questions that merit further investigation remain. In particular, one open question that arises is to close the gap between the upper and lower bounds provided here. We conjecture that the optimal complexity of monotone inclusion is: (i) $\Theta(\frac{LD}{\epsilon})$ when the operator is either $L$-Lipschitz or $\frac{1}{L}$-cocoercive, and (ii) $\Theta(\frac{L}{m}\log(\frac{L D}{\epsilon}))$ when the operator is $L$-Lipschitz and $m$-strongly monotone.  

\section*{Acknowledgements}
We thank Prof.~Ulrich Kohlenbach for useful comments and pointers to the literature. We also thank Howard Heaton for pointing out a typo in the proof of Lemma 2.1 in a previous version of this paper. 
\ifarxiv
\bibliographystyle{plainnat}
\fi
\bibliography{references}

\appendix
%
\section{Omitted Proofs}\label{appx:omitted-proofs}
%
%
\subsection{Unconstrained Setting with a Cocoercive Operator}
\begin{replemma}{lemma:c-k-non-inc}
    Let $\cc_k$ be defined as in Eq.~\eqref{eq:potential} and let $\vu^*$ be the solution to~\eqref{eq:monotone-inclusion} that minimizes $\|\vu_0 - \vu^*\|$. 
    Assume further that $\innp{F(\vu_1) - F(\vu_0), \vu_1 - \vu_0} \geq \frac{1}{L_1}\|F(\vu_1) - F(\vu_0)\|^2.$ If $A_{k+1}\cc_{k+1} \leq A_k\cc_k,$ $\forall k \geq 1,$ where $\{A_i\}_{i\geq 1}$ is a sequence of positive numbers that satisfies $A_1 = 1$, then:
    $$
       (\forall k \geq 1):\quad \|F(\vu_k)\| \leq L_{k}\frac{\lambda_k}{1-\lambda_k}{\|\vu_0 - \vu^*\|}.
    $$
\end{replemma}
\begin{proof}
The statement holds trivially if $\|F(\vu_k)\| = 0,$ so assume that $\|F(\vu_k)\|>0.$ 
Under the assumption of the lemma, we have that $A_k\cc_k \leq \cc_1,$ $\forall k \geq 1.$ From~\eqref{eq:halpern} and $\lambda_1 = \frac{1}{2}$, $\vu_1 = \vu_0 - \frac{1}{ L_1}F(\vu_0),$ and thus: 
$
    \cc_1 = \frac{1}{ L_1}\|F(\vu_1)\|^2 - \frac{1}{ L_1}\innp{F(\vu_1), F(\vu_0)}. 
$

Let $\vu^*$ be an arbitrary solution to~\eqref{eq:monotone-inclusion} (and thus also to~\eqref{eq:MVI}). As $\innp{F(\vu_1) - F(\vu_0), \vu_1 - \vu_0} \geq \frac{1}{ L_1}\|F(\vu_1) - F(\vu_0)\|^2$ and $\vu_1 = \vu_0 - \frac{1}{ L_1}F(\vu_0),$ it follows that $\|F(\vu_1)\|^2 \leq \innp{F(\vu_0), F(\vu_1)},$ and, thus $\cc_1 \leq 0.$ 
Further, as $A_k >0,$ we also have $\cc_k \leq 0,$ and, hence:
\begin{align*}
    \|F(\vu_k)\|^2 &\leq  L_k\frac{\lambda_k}{1-\lambda_k}\innp{F(\vu_k), \vu_0 - \vu_k}\\
    &=  L_k\frac{\lambda_k}{1-\lambda_k}\innp{F(\vu_k), \vu_0 - \vu^*} +  L_k\frac{\lambda_k}{1-\lambda_k}\innp{F(\vu_k), \vu^* - \vu_k}\\
    &\leq   L_k\frac{\lambda_k}{1-\lambda_k}\innp{F(\vu_k), \vu_0 - \vu^*} \leq  L_k\frac{\lambda_k}{1-\lambda_k}\|F(\vu_k)\|\cdot\|\vu_0 - \vu^*\|,
\end{align*}
where the last line is by $\vu^*$ being a solution to~\eqref{eq:MVI} and by the Cauchy-Schwarz inequality. 
The conclusion of the lemma now follows by dividing both sides of $\|F(\vu_k)\|^2 \leq  L_k\frac{\lambda_k}{1-\lambda_k}\|F(\vu_k)\|\cdot\|\vu_0 - \vu^*\|$ by $\|F(\vu_k)\|$ and observing that the statement holds for an arbitrary solution $\vu^*$ to \eqref{eq:monotone-inclusion}, and thus, it also holds for the one that minimizes the distance to $\vu_0.$ 
\end{proof}
\begin{replemma}{lemma:pot-dec}
    Let $\cc_k$ be defined as in Eq.~\eqref{eq:potential}. Let $\{A_i\}_{i \geq 1}$ be defined recursively as $A_1 = 1$ and $A_{k+1} = A_k \frac{\lambda_k}{(1-\lambda_k)\lambda_{k+1}}$ for $k \geq 1.$ Assume that $\{\lambda_i\}_{i\geq 1}$ is chosen so that $\lambda_1 = \frac{1}{2}$ and for $k \geq 1:$ $\frac{\lambda_{k+1}}{1-2\lambda_{k+1}} \geq  \frac{\lambda_k L_k}{(1-\lambda_k)L_{k+1}}$. Finally, assume that $L_k \in (0, \infty)$ and $\innp{F(\vu_k) - F(\vu_{k-1}), \vu_k - \vu_{k-1}}\geq \frac{1}{L_k}\|F(\vu_k) - F(\vu_{k-1})\|^2$, $\forall k.$ Then, 
    $$(\forall k \geq 1): \quad A_{k+1}\cc_{k+1} \leq A_k\cc_k.$$ 
\end{replemma}
\begin{proof}
By the assumption of the lemma, 
$$\frac{1}{L_{k+1}}\|F(\vu_{k+1}) - F(\vu_{k})\|^2\leq \innp{F(\vu_{k+1}) - F(\vu_{k}), \vu_{k+1} - \vu_{k}},$$ 
which, after expanding the left-hand side, can be equivalently written as:
\begin{equation}\notag
    \frac{1}{L_{k+1}}\|F(\vu_{k+1})\|^2 \leq \big\langle{F(\vu_{k+1}), \vu_{k+1} - \vu_k + \frac{2}{L_{k+1}}F(\vu_k)}\big\rangle - \big\langle{F(\vu_k), \vu_{k+1} - \vu_k + \frac{1}{L_{k+1}}F(\vu_k)}\big\rangle.
\end{equation}
From~\eqref{eq:halpern}, we have that $\vu_{k+1} - \vu_k = \frac{\lambda_{k+1}}{1-\lambda_{k+1}}(\vu_0 - \vu_{k+1}) - \frac{2}{L_{k+1}}F(\vu_k)$ and $\vu_{k+1} - \vu_k  = \lambda_{k+1}(\vu_0 - \vu_k) - \frac{2(1-\lambda_{k+1})}{L_{k+1}}F(\vu_k).$ Hence:
\begin{align*}
    \frac{1}{L_{k+1}}\|F(\vu_{k+1})\|^2 \leq &\; \frac{\lambda_{k+1}}{1-\lambda_{k+1}}\innp{F(\vu_{k+1}), \vu_0 - \vu_{k+1}} - \lambda_{k+1}\innp{F(\vu_k), \vu_0 - \vu_k}\\
    &+ \frac{1 - 2\lambda_{k+1}}{L_{k+1}}\|F(\vu_k)\|^2.
\end{align*}
Rearranging the last inequality and multiplying both sides by $A_{k+1},$ we have:
\begin{equation}\notag
\begin{aligned}
   &A_{k+1}\Big( \frac{1}{L_{k+1}}\|F(\vu_{k+1})\|^2 - \frac{\lambda_{k+1}}{1-\lambda_{k+1}}\innp{F(\vu_{k+1}), \vu_0 - \vu_{k+1}}\Big)\\
   &\hspace{1in}\leq  \frac{A_{k+1}(1 - 2\lambda_{k+1})}{L_{k+1}}\|F(\vu_k)\|^2 - A_{k+1}\lambda_{k+1}\innp{F(\vu_k), \vu_0 - \vu_k}.
\end{aligned}
\end{equation}
The left-hand side of the last inequality if precisely $A_{k+1}\cc_{k+1}.$ The right-hand side is $\leq A_k \cc_k,$ by the choice of sequences $\{A_i\}_{i \geq 1},$ $\{\lambda_i\}_{i\geq 1}.$ 
\end{proof}
%
\subsection{Operator Mapping}
\begin{repproposition}{prop:op-map-cocoercive}
Let $F$ be an $\frac{1}{L}$-cocoercive operator and let $G_\eta$ be defined as in Eq.~\eqref{eq:mon-op}, where $\eta \geq L.$ Then $G_\eta$ is $\frac{1}{2\eta}$-cocoercive. 
\end{repproposition}
\begin{proof}
As $\mathrm{Id} - \Pi_{\cu}$ is 1-cocoercive (by Fact~\ref{fact:firmly-ne}), we have, $\forall \vu, \vv \in E$:
\begin{align*}
    &\innp{\Pi_{\cu}\Big(\vu - \frac{1}{\eta}F(\vu)\Big) - \Pi_{\cu}\Big(\vv - \frac{1}{\eta}F(\vv)\Big), \vu - \frac{1}{\eta}F(\vu) - \Big(\vv - \frac{1}{\eta}F(\vv)\Big)}\\
    &\hspace{1in}= \innp{\frac{1}{\eta}(G_\eta(\vv) - G_{\eta}(\vu)) + \vu - \vv, \vu - \vv - \frac{1}{\eta}(F(\vu) - F(\vv))}\\
    & \hspace{1in} \geq \Big\|\frac{1}{\eta}(G_\eta(\vv) - G_{\eta}(\vu)) + \vu - \vv \Big\|^2.
\end{align*}
Hence:
\begin{equation}\label{eq:g-eta-Lip}
\begin{aligned}
    \frac{1}{\eta^2}\|G_\eta(\vv) - G_{\eta}(\vu)\|^2 \leq\;& \frac{1}{\eta}\innp{G_{\eta}(\vu) - G_{\eta}(\vv), \vu - \vv}\\
    &+ \frac{1}{\eta^2}\innp{G_{\eta}(\vu) - G_{\eta}(\vv), F(\vu) - F(\vv)}
    - \frac{1}{\eta}\innp{F(\vu) - F(\vv), \vu - \vv}.
\end{aligned}
\end{equation}
As $\eta \geq L$ and $F$ is $\frac{1}{L}$-cocoercive, $\frac{1}{\eta}\innp{F(\vu) - F(\vv), \vu - \vv} \geq \frac{1}{\eta^2}\|F(\vu) - F(\vv)\|^2.$ It remains to apply Young's inequality, which implies $\innp{G_{\eta}(\vu) - G_{\eta}(\vv), F(\vu) - F(\vv)} \leq \frac{1}{2}\|G_\eta(\vu) - G_{\eta}(\vv)\|^2 + \frac{1}{2}\|F(\vu) - F(\vv)\|^2.$
\end{proof}
\subsection{Approximating the Resolvent}
Let us start by proving the convergence of a version of the Extragradient method of~\cite{korpelevich1977extragradient} that does not require the knowledge of the Lipschitz constant $L$ (but does require knowledge of the strong monotonicity parameter $m$; when computing the resolvent we have $m = 1$). The algorithm is summarized in Algorithm~\ref{algo:eg+}.
\begin{algorithm}
\caption{EG Without the Knowledge of $L$}\label{algo:eg+}
\textbf{Input:} $a_0,\, \vu_0\in \cu,\, m, \epsilon.$ If not provided at the input or $> 1/m$, set $a_0 = 1/m.$ \;
\setcounter{AlgoLine}{0}

\nl$\vub_0 = \Pi_{\cu}(\vu_{0} - a_k F(\vu_{0}))$\;

\nl$k = 0$, $\delta_0 = \frac{a_0 m \epsilon}{5 \sqrt{2}}$\;

\nl\While{$\|\vub_k - \vu_k\|>\delta_k$}
{
\nl $k = k+1$, $a_k = a_{k-1}$\;

\nl $\vub_{k} = \Pi_{\cu}(\vu_{k} - a_k F(\vu_{k}))$\;

\nl $\vu_{k+1} = \argmin_{\vu \in \cu}\Big\{a_k \innp{F(\vub_{k}), \vu} + \frac{a_k m}{2}\|\vu - \vub_{k}\|^2 + \frac{1}{2}\|\vu - \vu_{k}\|^2\Big\}$ \label{line:update-uk-1}\;

\nl \label{Line:eg+step-cond}\While{$a_k\innp{F(\vub_k) - F(\vu_k), \vub_k - \vu_{k+1}} > \frac{1}{4}\|\vu_{k+1} - \vub_{k}\|^2 + \frac{1}{4}\|\vub_k - \vu_k\|^2$} 
{
\nl $a_k = \min\Big\{\frac{a_k}{2},\, \frac{\|\vub_k - \vu_k\|}{ \|F(\vub_k) - F(\vu_k)\|} \Big\}$ \label{line:eg+-min-ak}\;

\nl $\vub_{k} = \Pi_{\cu}(\vu_{k} - a_k F(\vu_{k}))$\;

\nl $\vu_{k+1} = \argmin_{\vu \in \cu}\Big\{a_k \innp{F(\vub_{k}), \vu} + \frac{a_k m}{2}\|\vu - \vub_{k}\|^2 + \frac{1}{2}\|\vu - \vu_{k}\|^2\Big\}$\label{line:update-uk-2}\;
}
\nl $\delta_k = \frac{a_k m \epsilon}{5 \sqrt{2}}$ \;
}
\Return $\vu_k$
\end{algorithm}
Observe that the update step for $\vu_k$ from Lines~\ref{line:update-uk-1} and~\ref{line:update-uk-2} can be written in the form of a projection onto $\cu;$ we chose to write it in the current form as it is more convenient for the analysis.

We now bound the convergence of Algorithm~\ref{algo:eg+}.
\begin{lemma}\label{lemma:eg+convergence}
 Let $a_0 > 0$ and let $F$ be $m$-strongly monotone and $L$-Lipschitz. Then, Algorithm~\ref{algo:eg+} outputs a point $\vu_k$ with $\|\vu_k - \vu^*\|\leq \epsilon$ after at most $k = O\big(\frac{L}{m}\log(\frac{L\|\vu_0 - \vu^*\|}{m\epsilon}\big))$ oracle queries to $F,$ where $\vu^*$ solves~\eqref{eq:SVI}. 
\end{lemma}
\begin{proof}
Define $A_k = \sum_{i=0}^k a_i.$ 
To prove the lemma, we will use the following gap (or merit) functions:
\begin{equation}\notag
    f_k = \frac{1}{A_k}\sum_{i=0}^k a_i \Big(\innp{F(\vub_i), \vub_i - \vu^*} - \frac{m}{2}\|\vub_i - \vu^*\|^2\Big).
\end{equation}
As $F$ is strongly monotone, $f_k \geq 0,\, \forall k.$ 
By convention, we take $f_{-1} = 0$ and $A_{-1}=0$, so that $A_k f_k - A_{k-1}f_{k-1} = a_k\Big(\innp{F(\vub_k), \vub_k - \vu^*} - \frac{m}{2}\|\vub_k - \vu^*\|^2\Big).$  Let us now bound $A_k f_k - A_{k-1}f_{k-1}$, and observe that $A_k f_k - A_{k-1}f_{k-1}\geq 0$. First, write
\begin{equation}\label{eq:eg+-change-in-merit}
\begin{aligned}
    A_k f_k - A_{k-1}f_{k-1} =&\; a_k\Big(\innp{F(\vub_k), \vub_k - \vu^*} - \frac{m}{2}\|\vub_k - \vu^*\|^2\Big)\\
    =&\; a_k\innp{F(\vub_k), \vu_{k+1} - \vu^*} + a_k\innp{F(\vu_k), \vub_k - \vu_{k+1}} \\
    &+ a_k\innp{F(\vub_k) - F(\vu_k), \vub_k - \vu_{k+1}} - \frac{a_k m}{2}\|\vub_k - \vu^*\|^2.
\end{aligned}
\end{equation}
By the first-order optimality of $\vu_{k+1}$ in its definition, we have, $\forall \vu:$
$$
    \innp{a_k F(\vub_k) + a_k m(\vu_{k+1} - \vub_{k}) + \vu_{k+1} - \vu_k, \vu - \vu_{k+1}} \geq 0,
$$
and, thus:
$$
    a_k \innp{F(\vub_k), \vu_{k+1} - \vu} \leq a_k m\innp{\vu_{k+1} - \vub_k, \vu - \vu_{k+1}} + \innp{\vu_{k+1} - \vu_k, \vu - \vu_{k+1}}.
$$
By the standard three-point identity (which can also be verified directly): 
$$
    \innp{\vu_{k+1} - \vu_k, \vu - \vu_{k+1}} = \frac{1}{2}\|\vu - \vu_k\|^2 - \frac{1}{2}\|\vu - \vu_{k+1}\|^2 - \frac{1}{2}\|\vu_k - \vu_{k+1}\|^2.  
$$
Thus, setting $\vu = \vu^*:$
\begin{align*}
    a_k \innp{F(\vub_k), \vu_{k+1} - \vu^*} = &\; a_k m\innp{\vu_{k+1} - \vub_k, \vu^* - \vu_{k+1}}\\
    &+ \frac{1}{2}\|\vu^* - \vu_k\|^2 - \frac{1}{2}\|\vu^* - \vu_{k+1}\|^2 - \frac{1}{2}\|\vu_k - \vu_{k+1}\|^2.
\end{align*}
Observe also that: 
$$
    \innp{\vu_{k+1} - \vub_k, \vu^* - \vu_{k+1}} = \frac{1}{2}\|\vu^* - \vub_k\|^2 - \frac{1}{2}\|\vu_{k+1} - \vub_k\|^2 - \|\vu^* - \vu_{k+1}\|^2.
$$
Thus, we have:
\begin{equation}\label{eq:eg+-1}
    \begin{aligned}
        a_k \innp{F(\vub_k), \vu_{k+1} - \vu^*} = &\; \frac{1}{2}\|\vu^* - \vu_k\|^2 - \frac{1 + a_k m}{2}\|\vu^* - \vu_{k+1}\|^2\\
        &- \frac{1}{2}\|\vu_k - \vu_{k+1}\|^2 + \frac{a_k m}{2}\|\vu^* - \vub_k\|^2 - \frac{a_km}{2}\|\vu_{k+1} - \vub_k\|^2.
    \end{aligned}
\end{equation}
By similar arguments:
\begin{equation}\label{eq:eqg+-2}
    a_k\innp{F(\vu_k), \vub_k - \vu_{k+1}} = \frac{1}{2}\|\vu_{k+1} - \vu_k\|^2 - \frac{1}{2}\|\vu_{k+1} - \vub_{k}\|^2 - \frac{1}{2}\|\vub_k - \vu_k\|^2.
\end{equation}
Combining Eq.~\eqref{eq:eg+-change-in-merit}-\eqref{eq:eqg+-2}:
\begin{equation}\notag
    \begin{aligned}
        A_k f_k - A_{k-1}f_{k-1} =&\; \frac{1}{2}\|\vu^* - \vu_k\|^2 - \frac{1 + a_k m}{2}\|\vu^* - \vu_{k+1}\|^2 + a_k\innp{F(\vub_k) - F(\vu_k), \vub_k - \vu_{k+1}}\\
        & - \frac{1 + a_k m}{2}\|\vu_{k+1} - \vub_{k}\|^2 - \frac{1}{2}\|\vub_k - \vu_k\|^2.
    \end{aligned}
\end{equation}
By the condition of the while loop in Line~\ref{Line:eg+step-cond} of Algorithm~\ref{algo:eg+}, and because $A_k f_k - A_{k-1}f_{k-1}\geq 0$,
\begin{equation}\label{eq:eg+-final-ineq}
    \begin{aligned}
        \frac{1 + a_k m}{2}\|\vu^* - \vu_{k+1}\|^2 + \frac{1 + 2a_k m}{4}\|\vu_{k+1} - \vub_{k}\|^2 + \frac{1}{4}\|\vub_k - \vu_k\|^2 \leq \frac{1}{2}\|\vu^* - \vu_k\|^2.
    \end{aligned}
\end{equation}
The condition of the while loop in Line~\ref{Line:eg+step-cond} of Algorithm~\ref{algo:eg+} is satisfied for any $a_k \leq \frac{1}{2L},$ as 
\begin{align*}
     a_k\innp{F(\vub_k) - F(\vu_k), \vub_k - \vu_{k+1}} &\leq a_k L \|\vub_k - \vu_k\|\cdot \|\vub_k - \vu_{k+1}\|\\
    &\leq \frac{a_k L}{2}\big(\|\vub_k - \vu_k\|^2 + \|\vub_k - \vu_{k+1}\|^2\big),
\end{align*}
where we have used the Cauchy-Schwarz inequality, the fact that $F$ is $L$-Lipschitz, and the Young inequality. Thus, in any iteration, $a_k > \frac{1}{4L},$ and the total number of times the while loop from Line~\ref{Line:eg+step-cond} is entered is at most  $\log_2(4L/a_0).$ 

From Eq.~\eqref{eq:eg+-final-ineq}, $\|\vu^* - \vu_{k+1}\|^2 \leq \frac{1}{1+m/(4L)}\|\vu^* - \vu_k\|^2 \leq (1- \frac{m}{8L})\|\vu^* - \vu_k\|^2.$ Thus, for any $\delta > 0,$ $\|\vu^* - \vu_k\| \leq \delta$ for $k \geq \frac{16L}{m}\log(\frac{\|\vu^* - \vu_0\|}{\delta}).$ Consequently, from Eq.~\eqref{eq:eg+-final-ineq}, $\|\vub_k - \vu_k\| \leq \sqrt{2} \delta$ whenever $\|\vu^* - \vu_k\| \leq \delta.$ In particular, for $\delta = \frac{a_k m \epsilon}{5 \sqrt{2}} \geq \frac{m \epsilon}{20\sqrt{2}L},$ $\|\vub_k - \vu_k\| \leq \sqrt{2}\delta = \frac{a_k m}{5}\epsilon$ after at most $k = \frac{16L}{m}\log(\frac{20\sqrt{2}L\|\vu^* - \vu_0\|}{m\epsilon}$ (outer loop) iterations. 

It remains to show that when $\|\vub_k - \vu_k \| \leq \delta,$ $\|\vu_k - \vu^*\| \leq \epsilon,$ and so Algorithm~\ref{algo:eg+} terminates.  Observe that $\vu_k - \vub_k = a_k G_{1/a_k}(\vu_k),$ where $G_{1/a_k}$ is the operator mapping defined in Eq.~\eqref{eq:op-mapping}. Thus, using Lemma~\ref{lemma:grad-mapping-approx} and noting that $a_k \leq 1/L_{\mathrm{loc}} = \frac{\|\vub_k - \vu_k\|}{\|F(\vub_k) - F(\vu_k)\|}$, if  $\|\vub_k - \vu_k\| \leq \frac{a_k m}{5}\epsilon$, we have  
$$
    \innp{F(\vub_k), \vub_k - \vu^*}\leq \frac{2 m}{5}\epsilon\|\vub_k - \vu^*\|.
$$
On the other hand, as $F$ is $m$-strongly monotone, we also have $\innp{F(\vub_k), \vub_k - \vu^*} \geq \frac{m}{2}\|\vub_k - \vu^*\|^2.$ Hence, $\|\vub_k - \vu^*\|\leq \frac{4\epsilon}{5}.$ Finally, applying the triangle inequality and as $a+k \leq 1/m:$
$$
    \|\vu_k - \vu^*\| \leq \|\vu_k - \vub_k\| + \|\vub_k - \vu^*\| \leq \frac{\epsilon}{5} + \frac{4\epsilon}{5} = \epsilon. 
$$
Note that we have already bounded the total number of inner and outer loop iterations. Observing that each inner iteration makes 2 oracle queries to $F$ and each outer iteration makes $2$ oracle queries to $F$ outside of the inner iteration, the bound on the total number of oracle queries to $F$ follows.  
\end{proof}
\begin{replemma}{lemma:approx-resolvent}
Let $\vub_k^* = J_{F+I_{\cu}}(\vu_k),$  %
where $\vu_k\in \cu$ and $F$ is $L$-Lipschitz. Then, there exists a parameter-free algorithm that queries $F$ at most $O((L+1)\log(\frac{(L+1)\|\vu_k - \vub_k^*\|}{\epsilon}))$ times and outputs a point $\vub_k$ such that $\|\vub_k - \vub^*_k\|\leq \epsilon.$ 
\end{replemma}
\begin{proof}
Observe first that $\vub^*_k$ solves~\eqref{eq:SVI} for operator $\bar{F}(\vu) = F(\vu) + \vu - \vu_k$ over the set $\cu.$ This follows from the definition of the resolvent, which implies:
$$
    \vub^*_k + F(\vub^*_k) + \partial I_{\cu}(\vub^*_k) \ni \vu_k.
$$
Equivalently: $\zeros \in  \bar{F}(\vub_k^*) + \partial I_{\cu}(\vub^*_k)$.

The rest of the proof follows by applying Lemma~\ref{lemma:eg+convergence} to $\bar{F},$ which is $(L+1)$-Lipschitz and $1$-strongly monotone. 
\end{proof}
%
%
%
\subsection{Inexact Halpern Iteration}
We start by first proving the following auxiliary result.
\begin{proposition}\label{prop:bnds-of-the-iterates}
Given an initial point $\vu_0 \in \cu,$ let $\vu_k$ evolve according to Eq.~\eqref{eq:inexact-halpern}, where $\lambda_k = \frac{1}{k+1}$. Then,
$$
   (\forall k \geq 1):\quad \|\vu_k - \vu^*\| \leq \|\vu_0 - \vu^*\| + \frac{1}{k + 1}\sum_{i=1}^k i \|\ve_{i-1}\|,
$$
where $\vu^*$ is such that $\|P(\vu^*)\| = 0.$ 
\end{proposition}
\begin{proof}
Let $T = \mathrm{Id} - P.$ Then $T(\vu^*) = \vu^*.$  By Fact~\ref{fact:nonexp-cocoerc-equiv}, $T$ is nonexpansive. Observe that we can equivalently write Eq.~\eqref{eq:inexact-halpern} as $\vu_k = \lambda_k \vu_0 + (1-\lambda_k)T(\vu_{k-1}) + (1-\lambda_k)\ve_{k-1}.$ Thus, using that $\vu^* = T(\vu^*)$:
\begin{align*}
    \|\vu_k - \vu^*\| &= \|\lambda_k (\vu_0 - \vu^*) + (1-\lambda_k)(T(\vu_{k-1}) - T(\vu^*)) + (1-\lambda_k)\ve_{k-1}\|\\
    &\leq \lambda_k\|\vu_0 - \vu^*\| + (1-\lambda_k)\|\vu_{k-1} - \vu^*\| + (1-\lambda_k)\|\ve_{k-1}\|,
\end{align*}
where we have used the triangle inequality and nonexpansivity of $T.$ The result follows by recursively applying the last inequality and observing that $\prod_{j=i}^k(1-\lambda_j) = \frac{i}{k+1}.$ 
\end{proof}
Using this proposition, we can now prove the following lemma.
\begin{replemma}{lemma:inexact-halpern}
 Let $\cc_k$ be defined as in Eq.~\eqref{eq:potential} with $P$ as the $\frac{1}{2}$-cocoercive operator, and let $L_k = 2,$ $\lambda_k = \frac{1}{k+1},$ and $A_k = \frac{k(k+1)}{2},$ $\forall k \geq 1$. If the iterates $\vu_k$ evolve according to~\eqref{eq:inexact-halpern} for an arbitrary initial point $\vu_0 \in \cu,$ then:  
    $$(\forall k \geq 1): \quad A_{k+1}\cc_{k+1} \leq A_k\cc_k + A_{k+1}\innp{\ve_k, (1-\lambda_{k+1})P(\vu_k) - P(\vu_{k+1})}.$$ 
Further, if, $\forall k \geq 1,$ $\|\ve_{k-1}\|\leq \frac{\epsilon}{4 k(k+1)},$ then $\|P(\vu_K)\|\leq \epsilon$ after at most $K = \frac{4\|\vu_0 - \vu^*\|}{\epsilon}$ iterations.
\end{replemma}
\begin{proof}
By the same arguments as in the proof of Lemma~\ref{lemma:c-k-non-inc}:
\begin{equation}\notag
    \frac{1}{2}\|P(\vu_{k+1})\|^2 \leq \innp{P(\vu_{k+1}), \vu_{k+1} - \vu_k + P(\vu_k)} - \innp{P(\vu_k), \vu_{k+1} - \vu_k + \frac{1}{2}P(\vu_k)}.
\end{equation}
From~\eqref{eq:inexact-halpern} and the definition of $\Tilde{P}$, we have that 
\begin{align*}
\vu_{k+1} - \vu_k &= \frac{\lambda_{k+1}}{1-\lambda_{k+1}}(\vu_0 - \vu_{k+1}) - P(\vu_k) - \ve_k, \text{and}\\
\vu_{k+1} - \vu_k  &= \lambda_{k+1}(\vu_0 - \vu_k) - {(1-\lambda_{k+1})}P(\vu_k) - {(1-\lambda_{k+1})}\ve_k.
\end{align*}
Hence:
\begin{align*}
    \frac{1}{2}\|P(\vu_{k+1})\|^2 \leq\;& \frac{\lambda_{k+1}}{1-\lambda_{k+1}}\innp{P(\vu_{k+1}), \vu_0 - \vu_{k+1}} - \lambda_{k+1}\innp{P(\vu_k), \vu_0 - \vu_k} \\
    &+ \frac{1 - 2\lambda_{k+1}}{2}\|P(\vu_k)\|^2 + \innp{\ve_k, (1-\lambda_{k+1})P(\vu_k) - P(\vu_{k+1})}.
\end{align*}
Plugging $\lambda_{k+1} = \frac{1}{k+2}$ in the last inequality and using the definition of $\cc_k$ and the choice of $A_k$ from the statement of the lemma completes the proof of the first part.

Using the same arguments as in the proof of Lemma~\ref{lemma:pot-dec}, we can conclude from $\quad A_{k+1}\cc_{k+1} \leq A_k\cc_k + A_{k+1}\innp{\ve_k, (1-\lambda_{k+1})P(\vu_k) - P(\vu_{k+1})},$ $\forall k \geq 1$ that:
\begin{equation}\label{eq:in-halp-err}
\begin{aligned}
    \frac{\|P(\vu_k)\|^2}{2} &\leq \frac{1}{k}\|P(\vu_k)\|\|\vu_0 - \vu^*\| + \frac{1}{A_k}\sum_{i=1}^k A_i \innp{\ve_{i-1}, (1-\lambda_i)P(\vu_{i-1}) - P(\vu_{i})}\\
    & = \frac{1}{k}\|P(\vu_k)\|\|\vu_0 - \vu^*\| + \frac{1}{k(k+1)}\sum_{i=1}^k i(i+1) \innp{\ve_{i-1}, \frac{i}{i+1}P(\vu_{i-1}) - P(\vu_{i})}.
\end{aligned}
\end{equation}
Let us now bound each $\innp{\ve_{i-1}, \frac{i}{i+1}P(\vu_{i-1}) - P(\vu_{i})}$ term. Recall that $P(\vu^*) = \zeros$ and $P$ is 2-Lipschitz (as discussed in Section~\ref{sec:prelims}, this follows from $P$ being $\frac{1}{2}$-cocoercive). Thus, we have:
\begin{align*}
    \innp{\ve_{i-1}, \frac{i}{i+1}P(\vu_{i-1}) - P(\vu_{i})} &= \innp{\ve_{i-1}, \frac{i}{i+1}(P(\vu_{i-1}) - P(\vu^*)) - (P(\vu_{i}) - P(\vu^*))}\\
    &\leq 2\|\ve_{i-1}\|\Big(\frac{i}{i+1}\|\vu_{i-1} - \vu^*\| + \|\vu_i - \vu^*\|\Big)\\
    &\leq 2\|\ve_{i-1}\|\Big(\frac{i+2}{i+1}\|\vu_0 - \vu^*\| + \frac{i}{i+1}\|\ve_{i-1}\| + \frac{2}{i+1}\sum_{j=1}^{i-1}j \|\ve_{j-1}\|\Big),
\end{align*}
where we have used Proposition~\ref{prop:bnds-of-the-iterates} in the last inequality. In particular, if $\|\ve_{i-1}\|\leq \frac{\epsilon}{4i(i+1)}$, then, $\forall i \geq 1$:
\begin{align*}
    \innp{\ve_{i-1}, \frac{i}{i+1}P(\vu_{i-1}) - P(\vu_{i})} \leq \frac{\epsilon}{2i(i+1)} \Big(\frac{i+2}{i+1}\|\vu_{0} - \vu^*\| + \epsilon/2\Big).
\end{align*}
Combining with Eq.~\eqref{eq:in-halp-err}:
\begin{equation}\label{eq:in-halp-err-final}
    \frac{\|P(\vu_k)\|^2}{2} \leq \frac{1}{k}\|P(\vu_k)\|\|\vu_0 - \vu^*\| + \frac{\epsilon}{2k}(\|\vu_0 - \vu^*\| + \epsilon/2).
\end{equation}
Observe that if $\|\vu_0 - \vu^*\|\leq \epsilon/2,$ as $P$ is 2-Lipschitz and $P(\vu^*) = \zeros,$ we would have $\|P(\vu_0)\| \leq \epsilon,$ and the statement of the second part of the lemma would hold trivially. Assume from now on that $\|\vu_0 - \vu^*\| > \epsilon/2.$ 
Suppose that $\|P(\vu_k)\| > \epsilon$ and $k \geq \frac{4\|\vu_0 - \vu^*\|}{\epsilon}.$  Then, dividing both sides of Eq.~\eqref{eq:in-halp-err-final} by $\|P(\vu_k)\|/2$ and using that $\|P(\vu_k)\| > \epsilon$ and $\|\vu_0 - \vu^*\| > \epsilon/2$, we get:
\begin{align*}
    \|P(\vu_k)\| < \frac{2\|\vu_0 - \vu^*\|(1+1/2)}{k} + \frac{2\cdot\epsilon/4}{k} <  \frac{3\epsilon}{4} + \frac{\epsilon}{4} \leq \epsilon,
\end{align*}
contradicting the assumption that $\|P(\vu_k)\| > \epsilon$ and completing the proof. 
\end{proof}
%
\subsection{Strongly Monotone Lipschitz Operators}
\begin{reptheorem}{thm:restarting-Lipschitz}
Given $F$ that is $L$-Lipschitz and $m$-strongly monotone, consider running the following algorithm $\mathcal{A}$, starting with $\vub_0\in \cu$:
\begin{center}
    $(\mathcal{A}):\quad$ At iteration $k,$ invoke Algorithm~\ref{algo:mon-Lip} with error parameter $\epsilon_k = \frac{7}{16}\|\Tilde{P}(\vu_{k-1})\|.$
\end{center}
Then, $\mathcal{A}$ outputs a point $\vu_k \in \cu$ with $\|P(\vu_k)\|\leq \epsilon$ after at most $\log_2(\|\vu_0 - \vu^*\|/\epsilon)$ iterations, where w.l.o.g.~$\epsilon \leq \frac{1}{2}$. The total number of oracle queries to $F$ until this happens is $O\big( (L + \frac{L}{m})\log(\|\vu_0 - \vu^*\|/\epsilon)\log(L + \frac{L}{m})\big).$ 
\end{reptheorem}
\begin{proof}
The first part of the theorem is immediate, as each call to Algorithm~\ref{algo:mon-Lip} ensures, due to Theorem~\ref{thm:mon-op}, that 
$$
\|P(\vu_k)\| \leq \frac{\|7\Tilde{P}(\vu_{k-1})\|}{16} \leq \frac{7\|{P}(\vu_{k-1})\|}{16} + \frac{\epsilon_k}{8} \leq \frac{\|{P}(\vu_{k-1})\|}{2},
$$ 
and $\|P(\vu_0)\| \leq 2\|\vu_0 - \vu^*\|$ as $P$ is 2-Lipschitz (because it is $\frac{1}{2}$-cocoercive) and $P(\vu^*) = \zeros.$

It remains to bound the number of calls to $F$ for each call to Algorithm~\ref{algo:mon-Lip}. Using Theorem~\ref{thm:mon-op} and $\|\Tilde{P}(\vu_k)\| = \Theta(\|P(\vu_k)\|)$, each call to Algorithm~\ref{algo:mon-Lip} takes $O(\frac{L\|\vu_{k-1} - \vu^*\|}{\|P(\vu_{k-1})\|} \log(\frac{L\|\vu_{k-1} - \vu^*\|}{\|P(\vu_{k-1})\|}))$ calls to $F.$ 
Denote $\vub^*_{k-1} = J_{F+\partial I_{\cu}}(\vu_{k-1}) = \vu_{k-1} - P(\vu_{k-1}).$ Using Proposition~\ref{prop:P-as-an-approx}:
$$
    \innp{F(\vub^*_{k-1}), \vub^*_{k-1} - \vu^*}\leq \|P(\vu_{k-1})\|\|\vub^*_{k-1} - \vu^*\|.
$$
On the other hand, as $F$ is $m$-strongly monotone and $\vu^*$ is an~\eqref{eq:MVI} solution, 
$$
    m \|\vub^*_{k-1} - \vu^*\|^2 \leq \innp{F(\vub_{k-1}^*), \vub^*_{k-1} - \vu^*}.
$$
Hence: $\|\vub^*_{k-1} - \vu^*\| \leq \frac{1}{m}\|P(\vu_{k-1})\|.$ It remains to use the triangle inequality and $P(\vu_{k-1}) = \vu_{k-1} - \vub^*_{k-1}$ to obtain:
\begin{equation}\label{eq:last-eq-for-m}
    \|\vu_{k-1} - \vu^*\| \leq \Big(1 + \frac{1}{m}\Big)\|P(\vu_{k-1})\|, 
\end{equation}
which completes the proof.
\end{proof}
%
\subsection{Lower Bounds}
We make use of the lower bound from~\cite{Ouyang2019} and the algorithmic reductions between the problems considered in previous sections to derive (near-tight) lower bounds for all of the problems considered in this paper. 

The lower bounds are for deterministic algorithms working in a (first-order) oracle model. For convex-concave saddle-point problems with the objective $\Phi(\vx, \vy)$ and closed convex feasible set $\cx \times \cy,$ any such algorithm $\mathcal{A}$ can be described as follows: in each iteration $k$, $\mathcal{A}$ queries a pair of points $(\bar{\vx}_k, \bar{\vy}_k) \in \cx \times \cy$ to obtain $(\nabla_\vx \Phi(\bar{\vx}_k, \bar{\vy}_k),\, \nabla_{\vy}\Phi(\bar{\vx}_k, \bar{\vy}_k)),$  and outputs a candidate solution pair $(\vx_k, \vy_k) \in \cx \times \cy.$ Both the query points pair $(\bar{\vx}_k, \bar{\vy}_k)$ and the candidate solution pair  $(\vx_k, \vy_k)$ can only depend on (i) global problem parameters (such as the Lipschitz constant of $\Phi$'s gradients or the feasible sets $\cx, \cy$) and (ii) oracle queries and answers up to iteration $k:$
$$
    \{\bar{\vx}_i,\, \bar{\vy}_i,\, \nabla_\vx \Phi(\bar{\vx}_i, \bar{\vy}_i),\, \nabla_{\vy}\Phi(\bar{\vx}_i, \bar{\vy}_i)\}_{i=0}^{k-1}.
$$

We start by summarizing the result from~\cite[Theorem 9]{Ouyang2019}.
\begin{theorem}\label{thm:lb-from-ouyang}
For any deterministic algorithm working in the first-order oracle model described above and any $L, R_{\cx},\, R_{\cy} > 0$, there exists a problem instance with a convex-concave function $\Phi(\vx, \vy): \cx \times \cy \to \rr$ whose gradients are $L$-Lipschitz, such that $\forall k = O(d):$  
$$
    \max_{\vy \in \vy} \Phi(\vx_k, \vy) - \min_{\vx\in \cx}\Phi(\vx, \vy_k) = \Omega\Big(\frac{L({R_{\cx}}^2 + R_{\cx}R_{\cy})}{k}\Big), 
$$
where $(\vx_k, \vy_k) \in \cx \times \cy$ is the algorithm output after $k$ iterations and $R_{\cx},\, R_{\cy}$ denote the diameters of the feasible sets $\cx, \, \cy,$ respectively, and where both $\cx, \, \cy,$ are closed and convex. 
\end{theorem}
The assumption of the theorem that $k = O(d)$ means that the lower bound applies in the high-dimensional regime $d = \Omega(\frac{L({R_{\cx}}^2 + R_{\cx}R_{\cy})}{\epsilon}),$ which is standard and generally unavoidable.

In the setting of VIs, we consider a related model in which an algorithm has oracle access to $F$ and refer to it as the operator oracle model. Similarly as for the saddle-point problems, we consider deterministic algorithms that on a given problem instance described by $(F, \cu)$ operate as follows: in each iteration $k$ the algorithm queries a point $\vub_k \in \cu$, receives $F(\vub_k),$ and outputs a solution candidate $\vu_k \in \cu$. Both $\vu_k$ and $\vub_k$ can only depend on (i) global problem parameters (such as the feasible set $\cu$ and the Lipschitz parameter of $F$), and (ii) oracle queries and answers up to iteration $k:$  
$
    \{\vub_i, F(\vub_i)\}_{i=0}^{k-1}.
$ 
Note that all methods described in this paper and most of the commonly used methods for solving VIs, such as, e.g., the mirror-prox method of~\cite{nemirovski2004prox} and dual extrapolation method of~\cite{nesterov2007dual}, work in this oracle model.
\begin{replemma}{lemma:lower-bounds}
For any deterministic algorithm working in the operator oracle model described above and any $L, D >0$, there exists a VI described by an $L$-Lipschitz-continuous operator $F$ and a closed convex feasible set $\cu$ with diameter $D$ such that:
\begin{enumerate}[label=(\alph*)]
    \item For all $\epsilon>0$ such that $k = \frac{L D^2}{\epsilon} = O(d)$, $\max_{\vu \in \cu} \innp{F(\vu_k), \vu_k - \vu} = \Omega(\epsilon)$;
    \item For all $\epsilon>0$ such that $k = \frac{L D}{\epsilon} = O(d)$, $\max_{\vu \in \{\cu\cap \mathcal{B}_{\vu_k}\}} \innp{F(\vu_k), \vu_k - \vu} = \Omega(\epsilon)$; 
    \item  If $F$ is $\frac{1}{L}$-cocoercive, then for all $\epsilon>0$ such that $k = \frac{L D}{\epsilon \log(D/\epsilon)} = O(d)$, it holds that  $$\max_{\vu \in \{\cu\cap \mathcal{B}_{\vu_k}\}} \innp{F(\vu_k), \vu_k - \vu} = \Omega(\epsilon)$$
    \item If $F$ is $m$-strongly monotone, then for all $\epsilon>0$ such that $k = \frac{L}{m} = O(d)$, it holds that  $$\max_{\vu \in \{\cu\cap \mathcal{B}_{\vu_k}\}} \innp{F(\vu_k), \vu_k - \vu} = \Omega(\epsilon).$$ 
\end{enumerate}
\end{replemma}
\begin{proof}$ $\newline
\noindent\textbf{Proof of~\ref{item:lb-SVI}:} 
Suppose that this claim was not true. Then we would be able to solve any instance with $L$-Lipschitz $F$ and $\cu$ with diameter bounded by $D$ and obtain $\vu_k$ with $\max_{\vu \in \cu}\innp{F(\vu_k), \vu_k - \vu} \leq \epsilon$  in $o(\frac{LD^2}{\epsilon})$ iterations, assuming the appropriate high-dimensional regime. In particular,  given any fixed convex-concave $\Phi(\vx, \vy)$ with $L$-Lipschitz gradients and feasible sets $\cx, \cy$ whose diameter is $D/2,$ let $\vu = [\subalign{\vx\\ \vy}],$ $F(\vu) = [\subalign{\nabla_{\vx}\Phi(\vx, \vy)\\ -\nabla_{\vy}\Phi(\vx, \vy)}]$, $\cu = \cx \times \cy.$ Then, it is not hard to verify that $F$ is monotone and $L$-Lipschitz (see, e.g.,~\cite{nemirovski2004prox,facchinei2007finite}) and the diameter of $\cu$ is $D.$ Thus, by assumption, we would be able to construct a point $\vu_k = [\subalign{\vx_k\\ \vy_k}]$ for which $\max_{\vu \in \cu}\innp{F(\vu_k), \vu_k - \vu} \leq \epsilon$  in $o(\frac{LD^2}{\epsilon})$ iterations. But then, because $\Phi$ is convex-concave, we would also have, for any $\vx \in \cx, \vy\in \cy$:
\begin{align*}
    \Phi(\vx_k, \vy_k) - \Phi(\vx, \vy_k) &= \max_{\vy \in \cy}\Phi(\vx_k, \vy_k) - \Phi(\vx_k, \vy_k) + \Phi(\vx_k, \vy_k) - \min_{\vx \in \cx}\Phi(\vx, \vy_k)\\
    &\leq \innp{\nabla_{\vy} \Phi(\vx_k, \vy_k), \vy - \vy_k} + \innp{\nabla_{\vx}\Phi(\vx_k, \vy_k), \vx_k - \vy_k} = \innp{F(\vu_k), \vu_k - \vu}.
\end{align*}
In particular, we would get:
$$
    \max_{\vy \in \cy}\Phi(\vx_k, \vy_k) - \min_{\vx \in \cx}\Phi(\vx, \vy_k) \leq \max_{\vu\in \cu}\innp{F(\vu_k), \vu_k - \vu} \leq \epsilon.
$$
Because we obtained this bound for an arbitrary $L$-Lipschitz convex-concave $\Phi$ and arbitrary feasible sets $\cx, \cy$ with diameters $D/2,$ Theorem~\ref{thm:lb-from-ouyang} leads to a contradiction.

\noindent\textbf{Proof of~\ref{item:lb-MI-Lip}:} 
If~\ref{item:lb-MI-Lip} was not true, then we would be able to obtain a point $\vu_k$ with $$\max_{\vu \in \{\cu\cap \mathcal{B}_{\vu_k}\}} \innp{F(\vu_k), \vu_k - \vu} = o(\epsilon/D)$$
in $k = \frac{LD^2}{\epsilon}$ iterations. But the same point would satisfy $\max_{\vu \in \cu} \innp{F(\vu_k), \vu_k - \vu} = o(\epsilon),$ which is a contradiction, due to~\ref{item:lb-SVI}.

\noindent\textbf{Proof of~\ref{item:MI-lb-coco}:} We prove the claim for $L = 2.$ This is w.l.o.g., due to the standard rescaling argument: if $F$ is $\frac{1}{L}$-cocoercive, then $\bar{F} = F/(2L)$ is $\frac{1}{2}$-cocoercive. Further, if, for some $\vu_k \in \cu,$ $$\max_{\vu \in \{\cu\cap \mathcal{B}_{\vu_k}\}} \innp{\bar{F}(\vu_k), \vu_k - \vu} = \Omega(\epsilon),$$ then $\max_{\vu \in \{\cu\cap \mathcal{B}_{\vu_k}\}} \innp{F(\vu_k), \vu_k - \vu} = \Omega(L\epsilon).$

Suppose that the claim was not true for a $\frac{1}{2}$-cocoercive operator $F.$ Then for any $M$-Lipschitz monotone operator $G,$ we would be able to use the strategy from Section~\ref{sec:mon-Lip} to obtain a point $\vu_k$ with $$\max_{\vu \in \{\cu\cap \mathcal{B}_{\vu_k}\}} \innp{G(\vu_k), \vu_k - \vu} = o(\epsilon)$$ in $k = \frac{M D}{\epsilon}$ iterations. This is a contradiction, due to~\ref{item:lb-MI-Lip}.

\noindent\textbf{Proof of~\ref{item:MI-lb-str-mon}:}
Suppose that the claim was not true, i.e., that there existed an algorithm that, for any $m, L >0,$ could output $\vu_k$ with $\max_{\vu \in \{\cu\cap \mathcal{B}_{\vu_k}\}} \innp{\bar{F}(\vu_k), \vu_k - \vu} = \epsilon/2$ in $k = o(L/m)$ iterations, for any $m$-strongly monotone and $L$-Lipschitz operator. 
Then for any $L$-Lipschitz monotone operator $F$, we could apply that algorithm to $\bar{F}(\cdot) = F(\cdot) + \frac{\epsilon}{2D}(\cdot - \vu_0)$ to obtain a point $\vu_k$ with $\max_{\vu \in \{\cu\cap \mathcal{B}_{\vu_k}\}} \innp{\bar{F}(\vu_k), \vu_k - \vu} = \epsilon/2$ in $k = o(LD/\epsilon)$ iterations. But then we would also have:
\begin{align*}
   \max_{\vu \in \{\cu\cap \mathcal{B}_{\vu_k}\}} \innp{{F}(\vu_k), \vu_k - \vu} &=  \max_{\vu \in \{\cu\cap \mathcal{B}_{\vu_k}\}} \innp{\bar{F}(\vu_k) - \frac{\epsilon}{2D}(\vu_k - \vu_0), \vu_k - \vu} \leq \epsilon,
\end{align*}
which is a contradiction, due to~\ref{item:lb-MI-Lip}. 
\end{proof} 
\end{document}

%% file: prologue.tex
\ifarxiv
\usepackage{amsmath}
\usepackage{amssymb}
\usepackage{amsthm}
\usepackage[margin = 1in]{geometry}
\usepackage{natbib}
\usepackage[colorlinks,urlcolor=blue,citecolor=blue,linkcolor=blue]{hyperref}
\usepackage[ruled]{algorithm2e}
\fi

\usepackage{bm}
\usepackage{color}
\usepackage{enumitem}

\makeatletter
\newtheorem*{rep@theorem}{\rep@title}
\newcommand{\newreptheorem}[2]{%
\newenvironment{rep#1}[1]{%
 \def\rep@title{#2 \ref{##1}}%
 \begin{rep@theorem}}%
 {\end{rep@theorem}}}
\makeatother

\newcommand{\innp}[1]{\left\langle #1 \right\rangle}

\newcommand{\zeros}{\textbf{0}}
\newcommand{\vx}{\mathbf{x}}

\newcommand{\cx}{\mathcal{X}}
\newcommand{\cy}{\mathcal{Y}}
\newcommand{\cu}{\mathcal{U}}

\newcommand{\cc}{\mathcal{C}}

\newcommand{\vy}{\mathbf{y}}

\newcommand{\vv}{\mathbf{v}}
\newcommand{\ve}{\mathbf{e}}

\newcommand{\vu}{\mathbf{u}}
\newcommand{\vub}{\bar{\mathbf{u}}}

\newcommand{\rr}{\mathbb{R}}

\makeatletter
\def\mathcolor#1#{\@mathcolor{#1}}
\def\@mathcolor#1#2#3{%
  \protect\leavevmode
  \begingroup
    \color#1{#2}#3%
  \endgroup
}
\makeatother

\newcommand*{\vsepfbox}[1]{%
  \begingroup
    \sbox0{\fbox{#1}}%
    \setlength{\fboxrule}{0pt}%
    \mbox{\kern-\fboxsep\fbox{\unhbox0}\kern-\fboxsep}%
  \endgroup
}

\ifarxiv
\theoremstyle{plain} \numberwithin{equation}{section}
\newtheorem{theorem}{Theorem}[section]
\numberwithin{theorem}{section}

\newtheorem{lemma}[theorem]{Lemma}
\newtheorem{proposition}[theorem]{Proposition}

\theoremstyle{definition}

\newtheorem{remark}[theorem]{Remark}

\fi 

\newreptheorem{lemma}{Lemma}
\newreptheorem{proposition}{Proposition}
\newtheorem{fact}[theorem]{Fact}
\newreptheorem{theorem}{Theorem}

\DeclareMathOperator*{\argmin}{argmin}

\makeatletter
\newcommand{\subalign}[1]{%
  \vcenter{%
    \Let@ \restore@math@cr \default@tag
    \baselineskip\fontdimen10 \scriptfont\tw@
    \advance\baselineskip\fontdimen12 \scriptfont\tw@
    \lineskip\thr@@\fontdimen8 \scriptfont\thr@@
    \lineskiplimit\lineskip
    \ialign{\hfil$\m@th\scriptstyle##$&$\m@th\scriptstyle{}##$\hfil\crcr
      #1\crcr
    }%
  }%
}
\makeatother